\documentclass[12pt]{amsart}
\usepackage[margin=1in]{geometry}
\usepackage{dsfont}
\usepackage{tikz-cd}
\usepackage{microtype,hyperref,amsfonts,amssymb,amsthm,mathrsfs,amsrefs}
\numberwithin{equation}{section}
\usepackage{CJK}
\usepackage[abs]{overpic}
\usepackage{grffile}
\usepackage{graphicx}
\usepackage{color}
\usepackage{stmaryrd}
\usepackage{lmodern}

\usepackage[scr=rsfs]{mathalfa}
\usepackage{pgf,tikz,pgfplots}
\pgfplotsset{compat=1.14}
\usetikzlibrary{arrows}

\newtheorem{thm}{Theorem}[section]
\newtheorem{cor}[thm]{Corollary}
\newtheorem{lem}[thm]{Lemma}

\newtheorem{definition}[thm]{Definition}
\newtheorem{rmk}[thm]{Remark}

\newtheorem{prop}[thm]{Proposition}
\newtheorem{conj}[thm]{Conjecture}

\newcommand{\Z}{\mathbb{Z}}
\newcommand{\R}{\mathbb{R}}
\newcommand{\C}{\mathbb{C}}
\newcommand{\bS}{\mathbb{S}}

\newcommand{\CF}{\mathcal{F}}

\newcommand{\bL}{\mathbf{L}}

\newcommand{\bT}{\mathbf{T}}

\newcommand{\val}{\textrm{val}\,}

\newcommand{\Fuk}{\mathrm{Fuk}}

\newcommand{\one}{\mathbf{1}}

\newcommand{\bP}{\mathbb{P}}

\newcommand{\MF}{\mathrm{MF}}

\title[Affine structures and SYZ fibrations]
{On the Complex Affine Structures of SYZ Fibration of Del Pezzo Surfaces}
\author{Siu-Cheong~Lau}
\email{lau@math.bu.edu}
\address{Mathematics and Statistics Department, Boston University, 111 Cummington Mall, Boston MA 02215}

\author{Tsung-Ju~Lee}
\email{tjlee@cmsa.fas.harvard.edu}
\address{Center of Mathematical Sciences and Applications, 20 Garden St., Cambridge, MA 02138}

\author{Yu-Shen~Lin}
\email{yslin@bu.edu}
\address{Mathematics and Statistics Department, Boston University, 111 Cummington Mall, Boston MA 02215}

\date{\today}
\begin{document}
	\begin{abstract} Given any smooth cubic curve $E\subseteq \mathbb{P}^2$,
		we show that the complex affine structure of the special Lagrangian fibration of $\mathbb{P}^2\setminus E$ constructed by Collins--Jacob--Lin \cite{CJL} coincides with 
		the affine structure used in Carl--Pomperla--Siebert \cite{CPS} for constructing mirror. Moreover, we use the Floer-theoretical gluing method to construct a mirror using immersed Lagrangians, which is shown to agree with the mirror constructed by Carl--Pomperla--Siebert.
	\end{abstract}
	\maketitle
	\section{Introduction}
	  
	   Mirror symmetry is a duality between the symplectic geometry of a 
	   Calabi--Yau manifold \(X\) 
	   and the complex geometry of its mirror \(\check{X}\). 
	   With the help of mirror symmetry, one can achieve a lot of enumerative 
	   invariants of Calabi--Yau manifolds, which are a priori hard to compute. 
	   
	   To construct the mirror for a Calabi--Yau manifold, Strominger--Yau--Zaslow 
	   proposed the following conjectures \cite{SYZ}: First of all, a Calabi--Yau manifold \(X\) 
	   near the large complex structure limit admits a special Lagrangian fibration. 
	   This is one of the very few geometric descriptions of Calabi--Yau manifolds. 
	   Second, the mirror \(\check{X}\) of \(X\) can 
	   be constructed as the dual torus fibration of \(X\). 
	   Third, the Ricci-flat metric on $X$ is closed to the semi-flat metric, with corrections coming from the holomorphic discs with boundaries on special Lagrangian torus fibres. 
	   
	   For a long time, Strominger--Yau--Zaslow conjecture serves a guiding principle for mirror symmetry. 
	   Many of its implications are proved as the building blocks for understanding mirror symmetry. For instance, it provides a geometric way of realizing the homological mirror functor \cite{LYZ}. However, there is very few progress on the original conjecture itself. Only very few examples of special Lagrangian fibrations are known due to technical difficulties of knowing explicit form of Ricci-flat metric. From the conjecture, one need to know the Ricci-flat metric for the existence of special Lagrangian fibration. While the explicit form of the Ricci-flat metric would involve the correction from the holomorphic discs. To retrieve such information, one need to know the boundary conditions, which are provided by the special Lagrangian torus fibres. Thus, the special Lagrangian fibration, the Ricci-flat metric and the correction from holomorphic discs form an iron triangle and firmly linked to each other.
	   Actually, all the examples in the literature are either with respect to the flat metric or the hyperK\"{a}hler rotation of the holomorphic Lagrangian fibrations. Furthermore, one usually can only track the hyperK\"ahler manifold via Torelli type theorem after hyperK\"{a}hler rotation rather than writing down the explicit equation. 
	   
	   To get around the analytic difficulties, Kontsevich--Soibelman \cite{KS1}, Gross--Siebert \cite{GS1} developed the algebraic alternative to construct the mirror families using rigid analytic spaces. One takes the dual intersection complex $B$ of the maximal degenerate Calabi--Yau varieties, there is a natural integral affine structures with singularities on $B$. By studying the scattering diagrams on $B$, one can reconstruct the Calabi--Yau family near the large complex structure limit. It is a folklore theorem that the affine manifold $B$ is the base for the Strominger--Yau--Zaslow conjecture, while the support of the scattering diagrams are the projection of the holomorphic discs with boundaries on special Lagrangian torus fibres. There are many success of understanding mirror symmetry via this algebraic approach. 

      On the other hand, one can use Lagrangian Floer theory to construct mirrors and prove homological mirror symmetry.  Fukaya \cite{F0} has proposed family Floer homology which was further developed by Tu \cite{T4} and Abouzaid \cites{A2,A3}.  The family Floer mirror is constructed as the set of Maurer--Cartan elements for the $A_{\infty}$ structures of the Lagrangian torus fibres quotient by certain equivalences. As Lagrangian torus fibres bound Maslov index zero holomorphic discs, the Maurer--Cartan elements will jump and induces non-trivial gluing of charts. It is expected that such jumps behave the same way as the cluster transformations associate to the ones in the scattering diagram. 
      
      A symplectic realization of the SYZ mirror construction was first illustrated in some inspiring examples by Auroux \cite{A}.  Using symplectic geometry, the SYZ mirror construction was realized for toric Calabi--Yau manifolds \cite{CLL} by Chan, Leung and the first named author.  They have interesting mirror maps and Gromov-Witten theory.  The mirror construction for blowing-up of toric hypersurfaces was realized by Abouzaid-Auroux-Katzarkov \cite{AAK}.  Fukaya--Oh--Ohta--Ono \cites{FOOO-T1,FOOO-T2,FOOO-MS} developed the Floer-theoretical construction in great detail for compact toric manifolds, which generalize and strengthen the result of Cho--Oh \cite{CO} for toric Fano manifolds.  
      
      In all these cases, the mirrors constructed in symplectic geometry coincide with the ones produced from Gross--Siebert program.  The holomorphic discs can be written down explicitly and no scattering of Maslov index zero discs occur.

	Singular SYZ fibers are the sources of Maslov index zero holomorphic discs and quantum corrections.  In \cites{CHL,CHL-nc}, Cho, Hong and the first named author found a way to construct a localized mirror of a Lagrangian immersion by solving the Maurer--Cartan equation for the formal deformations coming from immersed sectors.  Moreover, gluing between the local mirror charts based on Fukaya isomorphisms was developed in \cite{CHL-glue}.  Applying to singular fibers, it gives a canonical (partial) compactification of the SYZ mirror by gluing the local mirror charts of singular fibers with those of regular tori \cite{HKL}.
	 
    In general, it is difficult to explicitly compute the Floer theoretical mirror.  Maslov index zero discs can glue to new families of Maslov index zero discs, which is analogue of scattering or wall-crossing in Gross--Siebert program. It is in general complicated to control the scattering of Maslov index zero discs.
    
   	With the assumption that the Lagrangian fibration is special, one can have extra control of the locus of torus fibres bounding holomorphic discs.  They form affine lines with respect to the complex affine structure. In particular, this allows us to study a version of open Gromov--Witten invariants defined by the third author and identified them with the  
	tropical disc counting \cites{L1,L2,L14}. 
	   
	   It is reasonable to expect that the Gross--Siebert mirror and the Floer-theoretical mirror are equivalent. The first step toward such statement is  to identify the affine manifolds with singularities of the SYZ fibration and the one used in the Gross--Siebert program. 
	   \begin{conj}\label{1}
	   	   Let $X_t$ be a family of Calabi--Yau toric degeneration $X_0$ and $X_t$ admits a special Lagrangian fibration. Then the limit of the complex affine structures of the special Lagrangian fibration coincides with the affine structures on the dual intersection complex of $X_0$.
	   \end{conj}
	   In this paper, we will establish first such a  statement for the case of $\mathbb{P}^2$. 
	   \begin{thm}[={\bf Theorem} \ref{main thm}]
	   	  Conjecture \ref{1} holds for the SYZ fibration of $X=\mathbb{P}^2\setminus E$, where $E$ is a smooth cubic curve. 
	   \end{thm}
   	  
   	  The Gross--Siebert type mirror construction of $\mathbb{P}^2\setminus E$ is done by Carl--Pomperla--Siebert \cite{CPS} and the mirror is the fiberwise compactification of its Landau--Ginzburg mirror. In particular, it has the following description: First, take the toric variety $\bP^2/\Z_3$, whose moment-map polytope is dual to that of $\bP^2$, see Figure \ref{figure:dual-P2}.  We have the meromorphic function $W=z+w+1/zw$ on $\bP^2/\Z_3$. The pole divisor of $W$ is the sum of the three toric divisors.  The zero divisor of $W$ intersects with the pole divisor at three points. We blow up $\bP^2/\Z_3$ at these three points, so that $W$ induces an elliptic fibration.  (We can further blow up the three orbifold points of $\bP^2/\Z_3$ to make the total space smooth.)  Finally we delete the strict transform of the three toric divisors (which is the fiber at $\infty$) and this defines the mirror space. The Landau--Ginzburg superpotential is the elliptic fibration map induced by $W$. It is also worth noticing that the theorem is also achieved by Pierrick Beausseau with a different approach \cite{P}. We refer the readers for the inspiring heuristic discussion there about such an expectation from a different point of view.
   	  
	\begin{figure}
	\centering
	\includegraphics[scale=0.4]{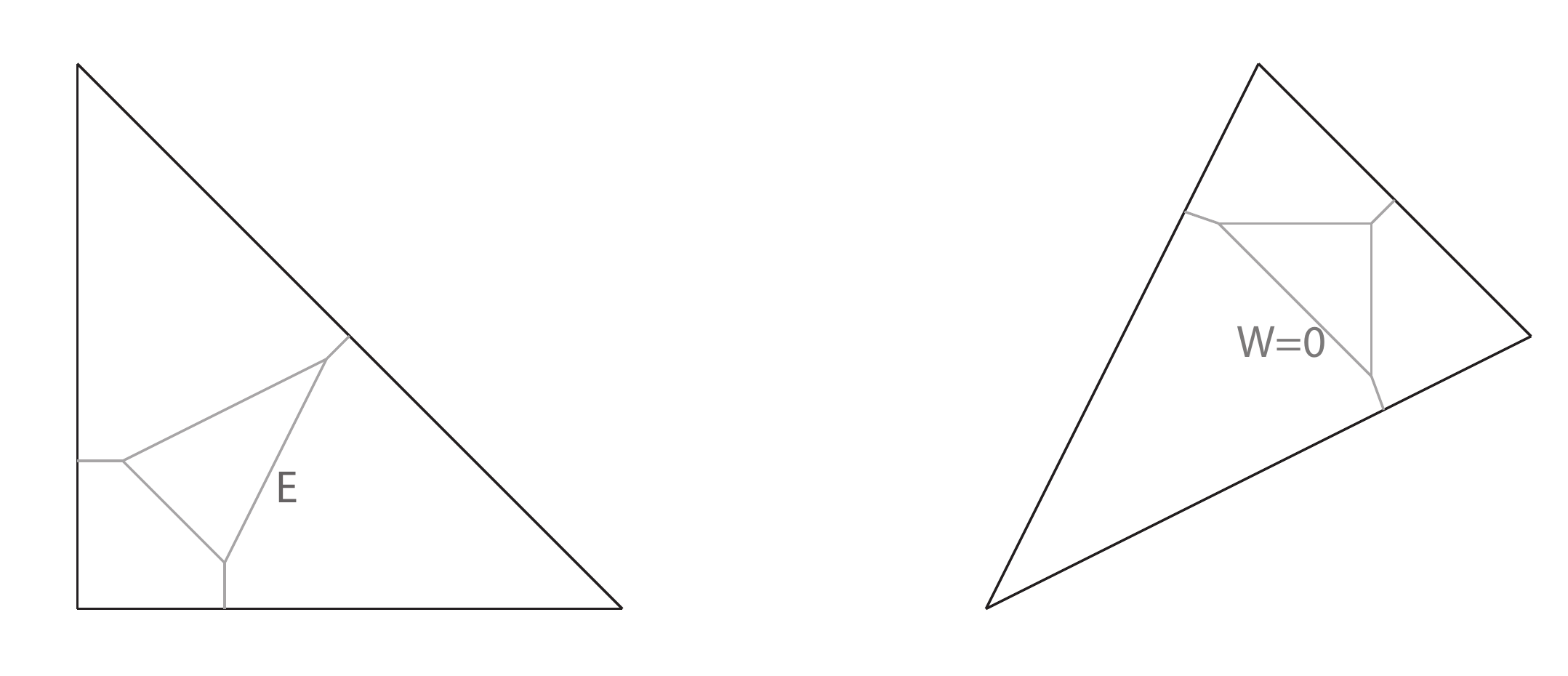}
	\caption{The moment-map polytope of $\bP^2$ and its dual.}
	\label{figure:dual-P2}
	\end{figure}   	  
   	  
      For the family Floer mirror, it is glued from torus charts, which are the deformation spaces of Lagrangian torus fibers.  Due to scattering of Maslov zero holomorphic discs, there are infinitely many walls and chambers in this case, and each chamber corresponds to a torus chart.
      
      On the other hand, in the Fano situation of this paper, we can use the method in \cites{CHL-glue,HKL} to construct a $\C$-valued mirror.  
      The special Lagrangian fibration on $\mathbb{P}^2\setminus E$ \cite{CJL} has three singular fibers which are nodal tori.
      Instead of the (infinitely many) torus fibers, we take the monotone moment-map torus together with three monotone Lagrangian immersions (in place of the singular SYZ fibers), and glue their deformation spaces together to construct the mirror.
      
      \begin{thm} \label{thm:agree}
      	For $\bP^2\setminus E$, the Floer-theoretical mirror glued from the deformation spaces of the monotone moment-map torus and the three monotone Lagrangian immersions coincides with the Carl--Pomperla--Siebert mirror described above.
      \end{thm}
      
      More precisely, the gluing construction has to be carried out over the Novikov field
      $$ \Lambda := \left\{\sum_{i=0}^\infty a_i \bT^{A_i} \mid a_i \in \C, A_i\in \R \textrm{ and increases to } +\infty \right\}$$
      so that the Lagrangian deformation spaces have the correct topology and dimension.  See Remark \ref{rmk:C}.  After we glue up a space over $\Lambda$ using Lagrangian Floer theory, we restrict to $\C$ to get a $\C$-valued mirror.
    
	   \subsection*{Outline of the paper} 
	    In Section \ref{section:SYZ-fibration-on-del-pezzo-surfaces}, we review the geometry of the special Lagrangian fibration on $\mathbb{P}^2\setminus E$ and the complex affine structure induced from the special Lagrangian fibration in Section . We also describe the affine manifold which is used to construct for mirror in \cite{CPS}. In Section \ref{section: proof}, we first explain how to use hyperK\"ahler rotation to reduce the problem to relative periods of an extremal rational elliptic surface, where the geometry can be very explicit. Then we verified various properties of the relative periods for the proof of the main theorem. In Section \ref{section: Floer mirror}, we carry out the Floer theoretical construction and show that it agrees with Carl--Pomperla--Siebert mirror.
	   	    
\section*{Acknowledgement}
The authors would like to thank S.-T.~Yau for constant encouragement and the 
Center of Mathematical Sciences and Applications for the wonderful 
research environment.  The first author expresses his gratitude to Cheol-Hyun Cho, Hansol Hong and Yoosik Kim for the useful joint works. The third author wants to thank Peirrick Beasseau, Tristan Collins, Adam Jacob for related discussion. 
The first author is supported by Simons Collaboration Grant \#580648.
The second author is supported by the Center of Mathematical Sciences and Applications.
The third author is supported by Simons Collaboration Grant \#635846. 
	 
	\section{SYZ Fibration on Del Pezzo Surfaces}
	\label{section:SYZ-fibration-on-del-pezzo-surfaces}
	  We will first review the results in \cite{CJL}: 
	 Let $Y$ be a del Pezzo surface or a rational elliptic surface. $D\in |-K_Y|$ be a smooth anti-canonical divisor and $X=Y\setminus D$. There exists a meromorphic volume form $\Omega$ on with simple pole along $D$ which is unique up to a $\mathbb{C}^{\ast}$-scaling. Therefore, one can view $X$ as a log Calabi--Yau surface. Moreover, Tian--Yau proved the following theorem: 
	    \begin{thm}[\cite{TY1}]
	    There exists an exact complete Ricci-flat metric $\omega_{TY}$ on $X$.
	   \end{thm} We will assume that $2\omega_{TY}^2=\Omega\wedge \bar{\Omega}$ after a suitable scaling of $\Omega$. 
    \begin{definition} Let $X$ be a complex manifold with a holomorphic volume form $\Omega$ and a Ricci-flat metric $\omega$. 
    	A half dimensional submanifold $L$ is a special Lagrangian with respect to $(\omega,\Omega)$ if $\left.\omega\right|_L=0$ and $\left.\mathrm{Im}\Omega\right|_L=0$. 
    \end{definition}
   It is conjectured by Yau and also Auroux \cite{A2} that there exists a special Lagrangian fibration on $X$. The conjecture is proved by Colllins--Jacob--Lin earlier.
	\begin{thm}[\cite{CJL}] \label{1001}
	  The log Calabi--Yau surface $X$ admits a special Lagrangian fibration $\pi:X\rightarrow B_{\mathrm{SYZ}}$ with respect to $\omega_{TY}$. 
	\end{thm}
	  Although the proof of the existence of special Lagrangian fibration in \cite{CJL} still largely use the hyperK\"ahler structure, an important difference from the earlier examples is that one knows which complex structure can support the special Lagrangian fibration. Moreover, one can use algebraic geometry to understand the complex structure after the hyperK\"ahler rotation. 
    \begin{thm}[\cite{CJL}] 
    \label{3}
    With the above notation and $d=(-K_Y)^2$.
	Let $\check{X}$ denotes the underlying topological space of $X$ with K\"ahler form and holomorphic volume form 
	 \begin{align}\label{HKrot}
	 \check{\omega}&=\mathrm{Re}~\Omega \notag \\
	 \check{\Omega}&=\omega-\sqrt{-1}\cdot\mathrm{Im}~\Omega.
	 \end{align}
	  Then $\check{X}$ admits an elliptic fibration and compactifiation to a rational elliptic surface $\check{Y}$ by adding an $I_d$ singular fibre over $\infty$.  
	    \begin{equation*}
		    \begin{tikzcd}
	   	    &\check{X} \ar[d]\ar[r,hook] &(\check{Y},I_d) \ar[d]\\ 
	   	    &\mathbb{C}\ar[r,hook] & (\mathbb{P}^1,\infty)\\
	   	    \end{tikzcd}
	   	\end{equation*}
	\end{thm}
   From the asymptotic behavior of $\check{\Omega}$, one has 
   \begin{prop} \cite{CJL2}
   	  The holomorphic $2$-form $\check{\Omega}$ on $\check{X}$ coincide with the meromorphic $2$-form on $\check{Y}$ with simply pole along the fibre over $\infty$.
   \end{prop}

     In particular, the rational elliptic surface $\check{Y}$ has singular configuration $I_9I_1^3$ for the case $Y=\mathbb{P}^2$ \cite{CJL}. The extremal rational elliptic surfaces have no deformation and thus can be identified by explicit equation. In the case of $Y=\mathbb{P}^2$, $\check{X}$ can actually be realized as the fibrewise compactification of the Landau--Ginzburg mirror 
    \begin{align}
    \label{equation:potential-function-w}
    \begin{split}
        W\colon (\mathbb{C}^{\ast})^2 &\longrightarrow \mathbb{C} \\
        (t_{1},t_{2})&\mapsto t_{1}+t_{2}+\frac{1}{t_{1}t_{2}}.
    \end{split}
    \end{align}
    It is straight-forward to check that $W$ has three critical values $\lambda_0,\lambda_1,\lambda_2$ and the cross-ratio with $\infty$ is fixed. Thus, we may assume that $\lambda_i=3\zeta^i$, where $\zeta=\exp{(2\pi i/3)}$. The fibres of $W$ are three-punctured elliptic curves. By computing the global monodromy which is conjugating to 
    \begin{equation*}
    \begin{bmatrix} 1& 9 \\ 0 & 1  \end{bmatrix},
    \end{equation*} 
    the Lefschetz fibration $W\colon(\mathbb{C}^{\ast})^2\rightarrow \mathbb{C}$ can be compactified to such an extremal rational elliptic surface by adding three sections and an $I_9$-fibre at infinity. 
    
	  \begin{figure}
  	  \centering
	  \includegraphics[scale=0.8]{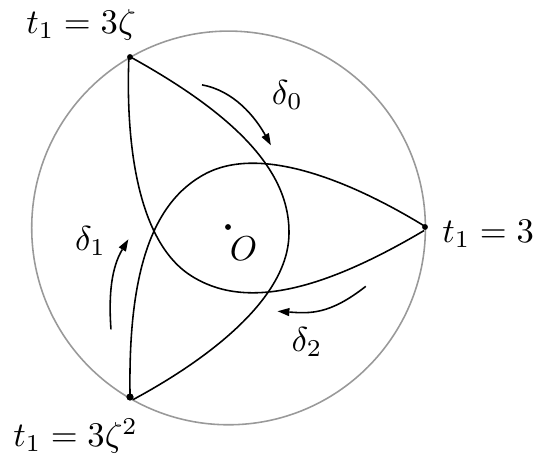}
  	  \caption{The vanishing cycles in \(E_{0}\)}
  	  \label{figure:auroux}
	  \end{figure}
  
 There is a $\mathbb{Z}_{3}$-action $(x,y)\mapsto (\zeta x,\zeta y)$ 
 on $(\mathbb{C}^{\ast})^2$ which induces a $\mathbb{Z}_{3}$-action on 
 the base $\mathbb{P}^1$ permuting the three critical values. 
 Let $E_0$ be the fibre over $0\in \mathbb{P}^1$ which is fixed 
 by the $\mathbb{Z}_{3}$-action. 
\begin{lem}[cf.~\cite{AKO}*{Lemma 3.1}]  
    \label{AKO}
    We can choose a basis \(\{a,b\}\) for $ \mathrm{H}_1(E_0,\mathbb{Z})\simeq \mathbb{Z}^2$
    and orientations for the vanishing cycles \([V_{0}]\), \([V_{1}]\),
    \([V_{2}]\) of \(\lambda_{0}\), \(\lambda_{1}\), \(\lambda_{2}\) such that 
    \([V_{0}]\), \([V_{1}]\) and \([V_{2}]\)
    are represented by \(-2a-b\), \(a+2b\) and \(a-b\) respectively 
    and the vanishing cycle from $\infty$ 
    along the curve \(\overline{\infty O}\) in
    \textsc{Figure~\ref{figure:orientation}} is represented by \(b\). 
    In particular, we have $[V_0]+[V_1]+[V_2]=0$. 
\end{lem}

\begin{rmk}
We remind the readers that our \(\lambda_{i}\) is different from the one used 
in \cite{AKO}. Indeed, \(V_{0}\equiv L_{0}\), \(V_{1}\equiv L_{2}\) 
and \(V_{2} \equiv -L_{1}\), where
\(L_{i}\) is the vanishing cycle defined in \cite{AKO}*{Lemma 3.1}.
However, to make the identification easier, we will use another basis.
\end{rmk}

We will describe the orientation explicitly in 
\S\ref{subsection:orientation} \textbf{(C)}.

Given a special Lagrangian fibration $X\rightarrow B_{\mathrm{SYZ}}$ with respect to $(\omega,\Omega)$, we will denote $L_q$ for the fibre over $q\in B_{\mathrm{SYZ}}$. Let $B_0$ be the complement of discriminant locus, then there exists an integral affine structure on $B_0$ \cite{H2} which we will now explain below: Choose a reference fibre $L_{q_0}$ and basis $e_i\in \mathrm{H}_1(L_{q_0})$. For a nearby torus fibre $L_q$ and a path $\phi$ connecting $q$ and $q_0$, let $C_i$ be the union of the parallel transport of $e_i$ along $\phi$. Then the complex affine coordinate $f_i(q)$ of $q$ is defined to be 
	    \begin{align} \label{2}
	       f_i(q):=\int_{C_i}\mbox{Im}\Omega,
	    \end{align} which is well-defined since $L_q, L_{q_0}$ are special Lagrangians. It is straight-forward to check that for a different choice of the basis and paths, the transition function falls in $\mathrm{GL}(n,\mathbb{Z})\rtimes \mathbb{R}^n$, where $n=\mbox{dim}_{\mathbb{R}}L_q$. Thus, $B_0$ is an integral affine manifold and we say $B_{\mathrm{SYZ}}$ is an integral affine manifold with singularities $\Delta=B_{\mathrm{SYZ}}\setminus B_0$. The above integral affine structure is usually known as the complex affine structure of the special Lagrangian fibration in the context of mirror symmetry. 
	    
\section{Equivalence of the Two Affine Structures}  \label{section: proof}
From now on, we concentrate on the case \(Y=\mathbb{P}^{2}\)
with the Landau--Ginzburg potential function \eqref{equation:potential-function-w}.
Recall that \(\mathbb{P}^{2}\) is defined by the polytope 
\(\Delta=\mathrm{Conv}\{(-1,-1),(2,-1),(-1,2)\}\). 
Let \(\nabla=\Delta^{\vee}\) be the dual polytope.
We denote by \(\mathbf{P}_{\nabla}\) the toric variety defined by \(\nabla\)
and by \(\widetilde{\mathbf{P}}_{\nabla}\to\mathbf{P}_{\nabla}\) 
the maximal projective crepant partial resolution of \(\mathbf{P}_{\nabla}\),
which is a resolution in the present case.

We denote by \(q\) the coordinate of the target space of the potential function \(W\)
in \eqref{equation:potential-function-w}
and regard \((W-q\cdot 1)\) as a \emph{holomorphic section} of 
the anti-canonical bundle over \(\widetilde{\mathbf{P}}_{\nabla}\). 
Precisely, the monomials \(t_{1}\), \(t_{2}\), 
\(t_{1}^{-1}t_{2}^{-1}\) correspond to the integral points \((1,0),(0,1),(-1,-1)\) in \(\nabla\) and
the monomial \(t_{1}^{0}t_{2}^{0}=1\) corresponds to the integral point \((0,0)\).
The subvariety \(\{W-q\cdot 1=0\}\subset \widetilde{\mathbf{P}}_{\nabla}\) 
gives the desired compactification of our fiber \(W^{-1}(q)\). 
The family \(\{W-q\cdot 1=0\}\) is a pencil spanned by the section \(1\) 
and \(t_{1}+t_{2}+t_{1}^{-1}t_{2}^{-1}\),
and can be extended to a family over \(\mathbb{P}^{1}\).
It is straightforward to check that the sections \(1\) and \(t_{1}+t_{2}+t_{1}^{-1}t_{2}^{-1}\) 
intersect at three points. Blowing-up the base locus gives a morphism \(\check{Y}\to \mathbb{P}^{1}\).
The fiber at \(\infty\in\mathbb{P}^{1}\) is a union of proper transforms of 
toric divisors in \(\widetilde{\mathbf{P}}_{\nabla}\),
which is a \(I_{9}\) fiber. For simplicity, the proper transform of the \(I_{9}\) fiber in \(\check{Y}\) 
is also denoted by \(I_{9}\).

Let \(\check{X}:=\check{Y}\setminus I_{9}\). First of all, it is clear that 
\begin{equation*}
\mathrm{H}_{4}(\check{Y},\mathbb{Z})\simeq\mathbb{Z},~
\mathrm{H}_{2}(\check{Y},\mathbb{Z})\simeq\mathbb{Z}^{10},~\mbox{and}~
\mathrm{H}_{0}(\check{Y},\mathbb{Z})\simeq\mathbb{Z}.
\end{equation*}
Secondly, from the Poincar\'{e} duality for orientable manifolds,
we have
\begin{equation*}
\mathrm{H}_{k}(\check{X},\mathbb{Z})\simeq\mathrm{H}^{4-k}_{\mathrm{c}}(\check{X},\mathbb{Z})
\simeq\mathrm{H}^{k}(\check{X},\mathbb{Z}),~\forall~k.
\end{equation*}
Finally, let \(U\) be the preimage of a small neighborhood around 
\(\infty\in\mathbb{P}^{1}\) under \(\check{Y}\to \mathbb{P}^{1}\).
\(I_{9}\) is a retract of \(U\). 
Utilizing the Mayer--Vietoris resolution for simple normal crossing varieties,
one can easily derive
\begin{equation*}
\mathrm{H}^{2}(I_{9},\mathbb{Z})\simeq \mathbb{Z}^{9},~
\mathrm{H}^{1}(I_{9},\mathbb{Z})\simeq \mathbb{Z},~\mbox{and}~
\mathrm{H}^{0}(I_{9},\mathbb{Z})\simeq \mathbb{Z}.
\end{equation*}
Consider the Mayer--Vietoris sequence assicoated to the pair \((U,\check{X})\), we can show that
\begin{equation*}
\mathrm{H}^{2}(\check{X},\mathbb{C})\simeq\mathrm{H}_{2}(\check{X},\mathbb{C})
\simeq\mathbb{C}^{2}
\end{equation*}
We put \(E_{q}=\{W-q\cdot 1=0\}\subset \check{Y}\). It follows that
\(\mathrm{H}_{2}(\check{X},\mathbb{Z})\)
is generated by the class of \(S^{1}\times S^{1}\subset (\mathbb{C}^{\ast})^{2}\) and 
the class of \(E_{q}\). 

From the construction, the standard toric form
\begin{equation*}
\frac{\mathrm{d}t_{1}}{t_{1}}\wedge\frac{\mathrm{d}t_{2}}{t_{2}}
\end{equation*}
on \((\mathbb{C}^{\ast})^{2}\) extends to a \emph{meromorphic} 
form on \(\widetilde{\mathbf{P}}_{\nabla}\) with poles along the union of 
toric divisors. Via the pullback further to \(\check{Y}\), 
we obtain a meromorphic \(2\)-form which
has poles exactly along \(I_{9}\).

In what follows, we set
\begin{equation*}
\check{\Omega} = \sqrt{-1}\cdot \frac{\mathrm{d}t_{1}}{t_{1}}\wedge\frac{\mathrm{d}t_{2}}{t_{2}}.
\end{equation*}
The meromorphic top form \(\check{\Omega}\) has the property that 
\begin{equation*}
\mathrm{Re}\left.\check{\Omega}\right|_{\mathrm{H}_{2}(\check{X},\mathbb{Z})} \equiv 0.
\end{equation*}
We can represent \(\check{\Omega}\) in a different way, 
which turns out to be useful in the sequel.
It follows from \eqref{equation:potential-function-w} that
\begin{align*}
\mathrm{d}q &= \mathrm{d}t_{1}+\mathrm{d}t_{2}-\frac{t_{1}\mathrm{d}t_{2}+t_{2}\mathrm{d}t_{1}}{(t_{1}t_{2})^{2}}\\
&= \left(1-\frac{1}{t_{1}^{2}t_{2}}\right)\mathrm{d}t_{1}
+\left(1-\frac{1}{t_{1}t_{2}^{2}}\right)\mathrm{d}t_{2}.
\end{align*}
A direct calculation gives
\begin{equation*}
\mathrm{d} q \wedge\mathrm{d}t_{1} = \frac{1-t_{1}t_{2}^{2}}{t_{2}}\frac{\check{\Omega}}{\sqrt{-1}}
\end{equation*}
and therefore
\begin{equation}
\label{equation:holomotphic-2-form-t1-q-coordinates}
\check{\Omega} = \sqrt{-1}\cdot\frac{t_{2}}{1-t_{1}t_{2}^{2}}\mathrm{d}q\wedge\mathrm{d}t_{1}
\end{equation}
provided that \(1-t_{1}t_{2}^{2}\ne 0\).

  \subsection{The affine structure of Carl--Pumperla--Siebert}
  \label{subsection:affine-manifolds-CPS}
  The construction of the mirror of a del Pezzo surface relative to a smooth anti-canonical divisor is studied in \cite{CPS}. 
  We describe the affine manifold with singularities used in \cite{CPS} below.
  The underlying space is $\mathbb{R}^2$ topologically. There are three singularities with local monodromy conjugate to 
  \begin{equation*}
  \begin{bmatrix} 1 & 1\\ 0 & 1 \end{bmatrix}
  \end{equation*} 
  locating at \(A'=(0,-1/2)\), \(B'=(1/2,1/2)\) and $C'=(-1/2,0)$. To cooperate with the standard affine structure of $\mathbb{R}^2$ for computation convenience, they introduce cuts and the affine transformation as follows: let
	\begin{align*}
	l_1^+&=\left\{\left(\frac{1}{2},y\right)\;\middle|\; y\geq \frac{1}{2} \right\}  \\
	l_1^-&=\left\{\left(x,\frac{1}{2}\right)\;\middle|\;x\geq \frac{1}{2} \right\}  \\
	l_2^+&=\left\{\left(x,-\frac{1}{2}\right)\;\middle|\;x\geq 0  \right\}          \\
	l_2^-&=\left\{\left(-t, -\frac{1}{2}-t\right)\;\middle|\; t\geq 0 \right\}          \\
	l_3^+&=\left\{\left(-\frac{1}{2}-t,-t\right)\;\middle|\; t\geq 0 \right\}  \\
	l_3^-&=\left\{\left(-\frac{1}{2}, y \right)\;\middle|\;y \geq 0 \right\}   
	\end{align*} 
	Disgard the sector bounded by $l_i^+,l_i^-$, then glue the cuts by the affine transformations 
	\begin{equation}
	\label{equation:monodromy-matrices-CPS}
	\begin{bmatrix} -1 & 4\\ -1 & 3 \end{bmatrix}, 
	\begin{bmatrix} 2 & 1\\ -1 & 0 \end{bmatrix}, 
	\begin{bmatrix} -1 & 1\\ -4 & 3 \end{bmatrix}
	\end{equation} 
	respectively and one reaches the affine manifold in \cite{CPS}. 
	See {\sc Figure \ref{figure:CPS}}.
	Notice that we glue the rays in clockwise order.

	  \begin{figure}
  	  \centering
	  \includegraphics[scale=0.85]{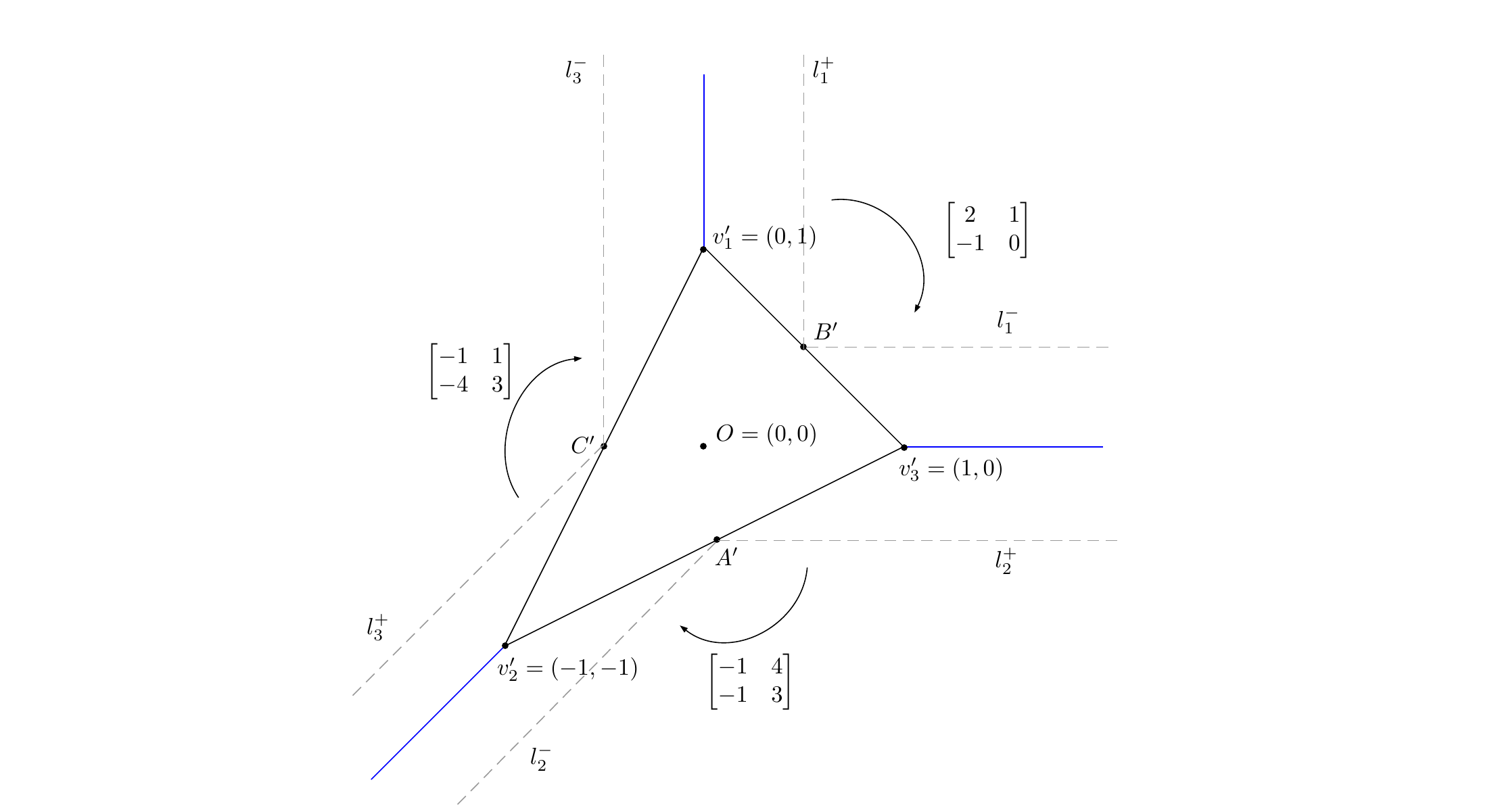}
  	  \caption{The affine plane given in \cite{CPS}}
  	  \label{figure:CPS}
	  \end{figure}
	
	\begin{definition}
	The affine manifold with singularities in Carl--Pumperla--Siebert 
	described above is denoted by \(B_{\mathrm{CPS}}\).
	\end{definition}	
	
 \subsection{A hyperK\"{a}hler rotation trick}
	To compute the complex affine coordinates, generally one needs to compute the relative periods for the imaginary part of the holomorphic volume form on $X$. Technically, it is not computable generally due to the fact that the special Lagrangian fibration is never explicit. The advantage of the work of Collins--Jacob--Lin is that one knows both the explicit equation of $X$ and $\check{X}$.  From (\ref{2}) and Theorem \ref{3}, one can compute the complex affine coordinates via the geometry on $\check{X}$ with a particular phase for $\check{\Omega}$. 
	From Mayer--Vietoris sequence, one has $\mathrm{H}_2(\check{X},\mathbb{Z})\cong \mathrm{H}_2(X,\mathbb{Z})\cong \mathbb{Z}^2$. Since $\omega_{TY}$ is exact, we have $\left.\omega_{TY}\right|_{\mathrm{H}_2(X,\mathbb{Z})}\equiv 0$. Because of the existence of the compact special Lagrangian tori, we have $\left.\Omega\right|_{\mathrm{H}_2(X,\mathbb{Z})}\neq 0$. Therefore, the phase of $\check{\Omega}$ need to be chosen such that 
	   \begin{align*}
	      \left.\mbox{Re}~\check{\Omega}\right|_{\mathrm{H}_2(\check{X},\mathbb{Z})}\equiv 0
	   \end{align*} from (\ref{HKrot}).
    To sum up, we have the following lemma:
    \begin{lem}
    	We resume the notation introduced in the last paragraph in 
    	\S\ref{section:SYZ-fibration-on-del-pezzo-surfaces}.
    	The complex affine structure of the special Lagrangian 
    	fibration in Theorem \ref{1001} for \(\mathbb{P}^2\) can be computed via 
    	   \begin{align*}
    	        f_i(q):=\int_{C_i}\mathrm{Im}~\check{\Omega}
    	   \end{align*}
    	 on the extremal rational elliptic surface $\check{Y}$ 
    	 with singular configuration $I_9I_1^3$. 
    	 Here $\check{\Omega}$ is the meromorphic volume form 
    	 on $\check{Y}$ with simple pole along the $I_9$ 
    	 fibre and $\mathrm{Re}~\check{\Omega}|_{\mathrm{H}_2(\check{X},\mathbb{Z})}\equiv 0$.
    \end{lem}
\subsection{Relative periods on \texorpdfstring{\(\check{Y}\)}{}}
We analyze the relative periods on \(\check{Y}\) in more detail in this paragraph.
In what follows, the fiber over \(q\) in \(\check{Y}\to\mathbb{P}^{1}\) is denoted by \(E_{q}\).
For simplicity, the points \(3\), \(3\zeta\) and \(3\zeta^{2}\) on 
\(\mathbb{C}_{q}:=\mathbb{P}^{1}\setminus\{\infty\}\) are denoted by \(A\), \(B\) and \(C\)
respectively.
\medskip
\paragraph{\bf (A) Vanishing cycles}
\label{subsection:cycle-notation}
Note that \(E_{3}\), \(E_{3\zeta}\) and \(E_{3\zeta^{2}}\) are all the
singular fibers in \(\check{Y}\to\mathbb{P}^{1}\) and each of which is of type \(I_{1}\).
According to \cite{AKO}, after a parallel transport to \(E_{0}\),
the vanishing cycle for \(E_{3\zeta^{j}}\) can be represented 
by a cycle \(V_{j}\) in \(E_{0}\) such that the image of \(V_{j}\) 
under the projection \((t_{1},t_{2})\mapsto t_{1}\) is
given by the arcs \(\delta_{j}\) drawn in {\sc Figure \ref{figure:auroux}}.

Note that the cycle class \([V_{j}]\in\mathrm{H}_{1}(E_{0},\mathbb{Z})\) 
is only defined up to sign at this moment
because we have not fixed the orientation yet.
However, once the orientation of \(V_{0}\) is determined, 
the \(\mathbb{Z}_{3}\)-action will uniquely determine 
the orientations for \(V_{1}\) and \(V_{2}\).

\begin{lem}
\label{lemma:orientation-compatible-z3-action}
Suppose the orientation of \(V_{j}\) is given with 
respect to the \(\mathbb{Z}_{3}\)-action. 
We can accordingly choose an integral basis \(\{c,d\}\subset\mathrm{H}_{1}(E_{0},\mathbb{Z})\)
such that \(\langle d,c\rangle=1\) and the vanishing 
cycles \([V_{0}]\), \([V_{1}]\) and \([V_{2}]\) are 
represented by \(c-2d\), \(c+d\) and \(-2c+d\), respectively. 
The presentations are chosen with respect to 
the \(\mathbb{Z}_{3}\)-action on the \(q\)-plane \(\mathbb{C}_{q}\) as well.
\end{lem}
\begin{proof}
Let \(\{c,d\}\) be a basis given by 
\begin{equation*}
\begin{cases}
a=-c+d\\
b=c
\end{cases}
\end{equation*}
where \(\{a,b\}\) is the basis in Lemma \ref{AKO}. We
get the desired basis.
\end{proof}

\medskip

\paragraph{\bf (B) Lefschetz thimbles}
For each \(j\), we define a simply connected domain
\begin{equation}
W_{j}:=\mathbb{C}\setminus \cup_{k\ne j} \{q:q=r\zeta^{k}~\mbox{with}~r\ge 3\}.
\end{equation}
For \(q\in W_{j}\), let \(\gamma\) be a smooth curve joint \(q\) and \(3\zeta^{j}\)
contained in \(W_{j}\). 
Let \(V_{j}\subset E_{0}\) be the representatives described in \textbf{(A)}.
Then the \emph{Lefschetz thimble} of \(V_{j}\)
along \(\gamma\), which is denoted by \(\Gamma^{\gamma}_{j}(q)\), is 
the union of the parallel transport of \(V_{j}\) along the cycle \(\gamma\).
Precisely,
\begin{equation}
\label{equation:lefschetz-thimble}
\Gamma^{\gamma}_{j}(q):=\cup_{q'\in\gamma} V_{j}^{(q')},
\end{equation}
where \(V_{j}^{(q')}\) is the parallel transport of \(V_{j}\) 
along any curve in \(W_{j}\) connecting \(0\) and \(q\)
and then from \(q\) to \(q'\) along \(\gamma\). 

We shall mention that different representatives \(\tilde{V}_{j}\) give
different Lefschetz thimbles \(\tilde{\Gamma}^{\gamma}_{j}(q)\) and
also that the Lefschetz thimble \emph{does} depend on
the choice of the curve connecting \(0\) and \(q\). One proves that 
in any case their difference is a \emph{coboundary}.
Consequently, by Stokes' theorem, the integral
\begin{equation}
\int_{\Gamma^{\gamma}_{j}(q)} \check{\Omega} 
\end{equation}
is independent of the choice of the representatives
and the curve connecting \(0\) and \(q\). However,
it is defined only up to a sign because of the 
orientation.

\begin{figure}
\centering
\includegraphics[scale=0.8]{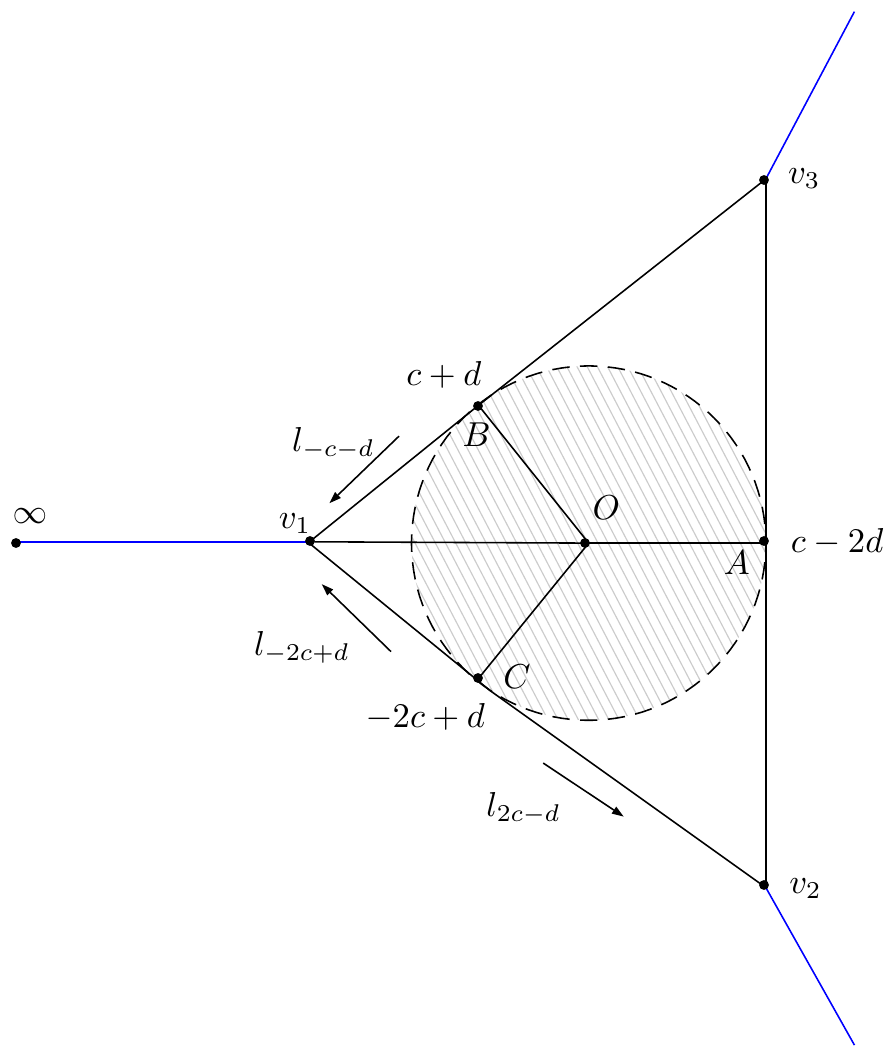}
\caption{The orientation of the cycles.
The oriented line segments \(\vec{\gamma}_{0}\), 
\(\vec{\gamma}_{1}\), and \(\vec{\gamma}_{2}\) 
defined in \eqref{equation:curves-OABC} are
the oriented line segments \(\protect\overrightarrow{AO}\), 
\(\protect\overrightarrow{BO}\) 
and \(\protect\overrightarrow{CO}\).
For other notation, see the paragraph {\bf (D)}.}
\label{figure:orientation}
\end{figure}

\medskip

\paragraph{\bf (C) Orientations}
\label{subsection:orientation}
Recall that \(\delta_{j}\) is an oriented arc with orientation drawn 
in {\sc Figure \ref{figure:auroux}}.
From \eqref{equation:potential-function-w}, we can solve
\begin{equation}
\label{equation:t2-equation-solve}
t_{2}^{\pm} = \left((q-t_{1}) \pm \sqrt{(q-t_{1})^{2}-4/t_{1}}\right)\slash 2.
\end{equation}
The orientation of \(V_{0}\) is chosen in the following manner.
First we note that \(V_{0}\) is set-theoretically equal to
the union of the graph of the 
holomorphic functions \(t_{2}^{+}\) and \(t_{2}^{-}\) along
the arc \(\delta_{0}\) defined in \eqref{equation:t2-equation-solve}. 
For the graph of \(t_{2}^{-}\), we take 
the induced orientation from \(\delta_{0}\). For the graph of \(t_{2}^{+}\), 
we shall take the induced orientation from \(-\delta_{0}\).
This pins down an orientation of \(V_{j}\).

For each \(j\), let 
\begin{equation}
\label{equation:curves-OABC}
\vec{\gamma}_{j}:=\left\{q\colon q=r\zeta^{j},~0\le r\le 3\right\}
\end{equation}
be an oriented (from \(3\zeta^{j}\) towards \(0\)) line segment 
(consult \textsc{Figure}~\ref{figure:orientation}). 

We denote by \(\vec{\Gamma}_{j}^{\gamma_{j}}(0)\) 
the Lefschetz thimble \(\Gamma_{j}^{\gamma_{j}}(0)\)
with the induced orientation from the \(S^{1}\)-bundle structure. 
Precisely, when restricting on \(\{q\colon q=r\zeta^{j},~0\le r<3\}\),
\(\Gamma_{j}^{\gamma_{j}}(0)\) becomes an \(S^{1}\)-bundle 
whose fibers are equipped with an orientation coming from \(V_{j}\). 
Therefore, it has an induced orientation which can be extended to 
the whole \(\Gamma_{j}^{\gamma_{j}}(0)\).

\begin{rmk}
It follows from the construction that the orientations for
\(\vec{\Gamma}_{j}^{\gamma_{j}}(0)\) are compatible with 
the \(\mathbb{Z}_{3}\)-action.
\end{rmk}

\begin{prop}
\label{proposition:orientation}
We have, for all \(j\), 
\begin{equation}
\int_{\vec{\Gamma}^{\gamma_{j}}_{j}(0)}\mathrm{Im}~\check{\Omega}\in\mathbb{R}_{+}.
\end{equation}
\end{prop}
\begin{cor}
\label{corollary:infinity-main}
We have
\begin{equation}
\lim_{q\to-\infty}\mathrm{Im} \int_{\vec{\Gamma}_{0}^{\gamma_{0}}(q)}\check{\Omega} =\infty.
\end{equation}
\end{cor}
\begin{proof}
We defer the proofs of Proposition \ref{proposition:orientation} 
and Corollary \ref{corollary:infinity-main}
in Appendix \ref{appendix:proof-of-proposition}.
\end{proof}

Once we pin down the orientation of \(V_{0}\), 
the orientations of \(V_{j}\) are also uniquely determined.
We then pick an integral basis \(\{c,d\}\subset\mathrm{H}_{1}(E_{0},\mathbb{Z})\)
such that the vanishing cycles \([V_{0}]\), \([V_{1}]\), and \([V_{2}]\) are 
represented by \(c-2d\), \(c+d\), and \(-2c+d\), respectively as in Lemma
\ref{lemma:orientation-compatible-z3-action}. 
\medskip
\paragraph{\bf (D) Affine structures}
\label{paragraph:d-affine-structure}
We describe the affine structure on \(\mathbb{C}\) using the elliptic fibration 
\(\check{X}\to\mathbb{C}\).
Consider the set
\begin{equation*}
\label{equation:affine-curve-passing-B}
\Xi_{1}:=\left\{q\in\mathbb{C}:\mathrm{Im} \int_{\Gamma_{1}^{\gamma}(q)} \check{\Omega} = 0,~
\gamma:\mbox{curve contained in \(\mathbb{C}\setminus\{A,C\}\) joining \(B\) and \(q\)}\right\}.
\end{equation*}
The condition that the imaginary part is equal to zero is independent 
of the choice of the curve \(\gamma\)
since the monodromy matrices are real. The set \(\Xi_{1}\) is well-defined.
Notice that the set \(\Xi_{1}\setminus\{B\}\) has two connected components.

Working on the simply connected domain \(V=W_{1}\cap W_{2}\cap W_{3}\),
we can define the locus
\begin{equation*}
l_{-c-d}:=\left\{\int_{\Gamma_{1}^{\gamma}(q)} \check{\Omega}\in\mathbb{R}_{+}~\mbox{for 
any parameterized curve \(\gamma\subset V\) joining \(B\) and \(q\)}\right\}.
\end{equation*}
Similarly we can define another curve \(l_{-2c+d}\subset\Xi_{2}\) by requiring that
\begin{equation*}
l_{-2c+d}:=\left\{\int_{\Gamma_{2}^{\gamma}(q)} \check{\Omega}\in\mathbb{R}_{-}~\mbox{for 
any parameterized curve \(\gamma\subset V\) joining \(C\) and \(q\)}\right\}.
\end{equation*}
\begin{rmk}
The above definition explains the notation in {\sc Figure \ref{figure:orientation}}.
However, we have not verified the validity of the intersection point \(v_{1}\) 
in {\sc Figure \ref{figure:orientation}}. 
We will prove that the curves \(\overline{O\infty}\), \(l_{-c-d}\)
and \(l_{-2c+d}\) intersect at one point,
where \(\overline{O\infty}\) denotes the negative real axis.
\end{rmk}

We resume the notation given in this subsection and in 
{\sc Figure \ref{figure:orientation}} and {\sc Figure \ref{figure:auroux}}
without recalling it.
To describe the affine structure a little bit more, we need
to study the integration over the Lefschetz thimbles.
\begin{definition}
Let \(q\in \overline{O\infty}\). We denote by
\(\vec{\sigma}_{0}(q)\) the oriented curve \(\overline{OA}\cup\overline{Oq}\)
from \(A\) toward \(q\),
\(\vec{\sigma}_{1}(q)\) the oriented curve \(\overline{OB}\cup\overline{Oq}\) 
from \(A\) toward \(q\) and
by \(\vec{\sigma}_{2}(q)\) the oriented curve \(\overline{OC}\cup\overline{Oq}\)
from \(A\) toward \(q\)
(cf.~\textsc{Figure~\ref{figure:affine-q-plane-curves}}).

For simplicity, we will write \(\Gamma_{j}(q):=\Gamma_{j}^{\vec{\sigma}_{j}(q)}(q)\)
(see \eqref{equation:lefschetz-thimble}), i.e., \(\Gamma_{0}(q)\) is the 
Lefschetz thimble of the cycle \(V_{0}\) along 
\(\vec{\sigma}_{0}(q)\),
\(\Gamma_{1}(q)\) is the Lefschetz thimble of the cycle 
\(V_{1}\) along \(\vec{\sigma}_{1}(q)\) and 
\(\Gamma_{2}(q)\) is the Lefschetz thimble of the cycle 
\(V_{2}\) along \(\vec{\sigma}_{2}(q)\).
\end{definition}

\begin{figure}
 	\centering
	\includegraphics[scale=0.9]{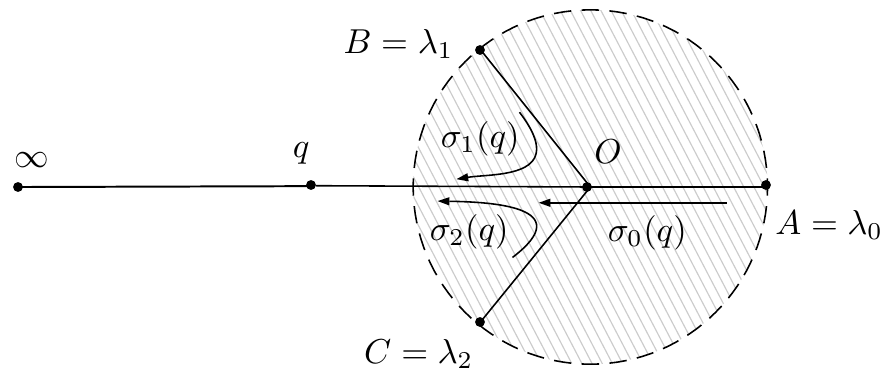}
  	\caption{The curves \(\sigma_{j}(q)\).}
   	\label{figure:affine-q-plane-curves}
\end{figure}
\begin{lem}\label{Z_2 symmetry}
The map \(\phi_0(q)=\bar{q}\) fixing \(O\) and \(A\) but exchanging \(B\) and \(C\) 
is an automorphism of the affine structure. 
\end{lem}

\begin{proof}
The affine lines are mapped to affine lines via \(\phi_{0}\).
\end{proof}
In particular, the lemma implies that the fixed locus of \(\phi_{0}\), 
an arc from \(A\) to infinity passing through \(O\) and another arc 
from \(A\) to infinity without passing through \(O\), are affine lines. 

\begin{cor}
\label{cor:linear-transform-complex-conjugation}
The induced map \((\phi_{0})_{\ast}\colon 
\mathrm{H}_{1}(E_{0},\mathbb{Z})\to\mathrm{H}_{1}(E_{0},\mathbb{Z})\) 
under the basis \(\{c,d\}\) is given by 
\begin{equation*}
\begin{bmatrix}
1 & 1\\
0 & -1
\end{bmatrix}.
\end{equation*}
\end{cor}
\begin{proof}
Note that \(\phi_{0}\) is the complex conjugation. We have
\((\phi_{0})_{\ast}(c-2d)=-c+2d\) and \((\phi_{0})_{\ast}(c+d)=2c-d\) by our choice of orientations,
which yields the corollary.
\end{proof}

\begin{lem}
\label{lemma:integration-cycle-real}
We have
\begin{equation*}
\int_{\Gamma_{1}(q)-\Gamma_{2}(q)} 
\check{\Omega}\in \mathbb{R},~\forall~q\in\overline{O\infty}.
\end{equation*}
\end{lem}
\begin{proof}
The lemma is proved by using the \(\mathbb{Z}_{2}\)-symmetry
\(\phi_{0}\). Since \(\Gamma_{1}(q)-\Gamma_{2}(q)\) is invariant under \(\phi_{0}\), 
we have
\begin{equation*}
\int_{\Gamma_{1}(q)-\Gamma_{2}(q)} \check{\Omega}
=\int_{(\phi^{-1}_{0})_{\ast}(\Gamma_{1}(q)-\Gamma_{2}(q))} \phi_{0}^{\ast}\check{\Omega}
=\int_{\Gamma_{1}(q)-\Gamma_{2}(q)} \overline{\check{\Omega}}
\end{equation*}
and the conclusion holds.
\end{proof}
We conclude this paragraph by proving the validity of the 
existence of the triple intersection point \(v_{1}\) in {\sc Figure \ref{figure:orientation}}.
Let us write
\begin{align}
\label{equation:definition-f1}
\begin{split}
F_{1}(q):=&\int_{\Gamma_{1}(q)} \mathrm{Im}~\check{\Omega}\\
=&\int_{\Gamma_{1}(q)-\Gamma_{1}(0)} \mathrm{Im}~\check{\Omega}
+ \int_{\Gamma_{1}(0)} \mathrm{Im}~\check{\Omega}\\
=& \int_{\Gamma_{1}(q)-\Gamma_{1}(0)} \mathrm{Im}~\check{\Omega} + 
\int_{\Gamma_{0}(0)}\mathrm{Im}~\check{\Omega}.
\end{split}
\end{align}
The last equality in \eqref{equation:definition-f1} holds
by the \(\mathbb{Z}_{3}\)-symmetry on the \(q\)-plane.

For \(q\le 0\), 
\begin{equation}
\Gamma_{1}(q)-\Gamma_{1}(0)=\cup_{q'\in \overline{qO}}~V_{1}^{(q')}.
\end{equation}
Consequently, utilizing Lemma \ref{lemma:integration-cycle-real} 
and the relation \(2V_{1}=-V_{0}-(V_{2}-V_{1})\), we get
\begin{align}
\begin{split}
\int_{\Gamma_{1}(q)-\Gamma_{1}(0)} \mathrm{Im}~\check{\Omega} &=
\int_{\cup_{q'\in \overline{qO}}~V_{1}^{(q')}} 
\mathrm{Im}~\check{\Omega}\\
&=-\frac{1}{2}\int_{\Gamma_{0}(q)-\Gamma_{0}(0)} 
\mathrm{Im}~\check{\Omega}
-\frac{1}{2}\int_{\Gamma_{2}(q)-\Gamma_{1}(q)}
\mathrm{Im}~\check{\Omega}\\
&=-\frac{1}{2}\int_{\Gamma_{0}(q)-\Gamma_{0}(0)} 
\mathrm{Im}~\check{\Omega},
\end{split}
\end{align} where the last equality comes from (\ref{lemma:integration-cycle-real}).
Then \eqref{equation:definition-f1} is transformed into
\begin{equation*}
F_{1}(q) = -\frac{1}{2} \int_{\Gamma_{0}(q)} 
\mathrm{Im}~\check{\Omega} + \frac{3}{2} \int_{\Gamma_{0}(0)}\mathrm{Im}~\check{\Omega},
~\mbox{for}~q\le 0.
\end{equation*}
From Corollary \ref{corollary:infinity-main}, we see that 
\(F_{1}(q)\to-\infty\) when \(q\to -\infty\).
Together with \(F_{1}(0)>0\), there exists some \(v_{1}\in \overline{O\infty}\) such that
\begin{equation*}
F_{1}(v_{1}) = 0.
\end{equation*}
Then \(v_{1}\) is the triple intersection point we are looking for.

	\subsection{Proof of the Main Theorem}
	In this paragraph, we will identify the affine manifold with singularities 
	\(B_{\mathrm{SYZ}}\) with \(B_{\mathrm{CPS}}\).
	Recall that the base $B_{\mathrm{SYZ}}$ of the special Lagrangian 
	fibration for $\mathbb{P}^2\setminus E$ can be topologically 
	identified with $\mathbb{C}$ and \(B_{\mathrm{CPS}}\) is the 
	affine manifolds constructed in \cite{CPS}. (See \S\ref{subsection:affine-manifolds-CPS}.)
	We will prove the following main theorem.
	\begin{thm}\label{main thm}
		 There exists an affine isomorphism between $B_{\mathrm{SYZ}}$ and $B_{\mathrm{CPS}}$. 
   \end{thm}		
	
    The affine manifold \(B_{\mathrm{SYZ}}\) has three singularities at 
    \(A\), \(B\) and \(C\). 
    Let us study their monodromies in more detail.

	 \begin{lem} \label{PL}
		Assume that the vanishing cycle has class \(pc+qd\in\mathrm{H}_{1}(E_{0},\mathbb{Z})\).
		Then the counter-clockwise monodromy across 
		the branch cut is 
		\begin{equation*}
		\begin{bmatrix}   1+pq & -p^2\\ q^2 & 1-pq \end{bmatrix}
		\end{equation*}
		with respect to the basis \(\{c,d\}\). 
	\end{lem}
	\begin{proof}
		From Picard--Lefschetz formula, the counterclockwise monodromy is given by
	\begin{equation}
	\mathrm{T}(v) = v + \langle\delta,v\rangle \delta
	\end{equation}
	where \(\delta\) is the vanishing cycle and 
	\(v\in\mathrm{H}_{1}(E_{0},\mathbb{Z})\).
	We can compute
	\begin{equation*}
	\begin{cases}
	\mathrm{T}(c) = c + \langle pc+qd,c\rangle(pc+qd)
	= (1-pq)c + q^{2} d \\
	\mathrm{T}(d) = d + \langle pc+qd,d\rangle(pc+qd)
	= -p^{2} c + (1-pq)d.
	\end{cases}
	\end{equation*}
	Notice that the monodromy is conjugate to 
	\begin{equation*}
	\begin{bmatrix} 1 & 1 \\0 & 1\end{bmatrix}
	\end{equation*}
	and the vanishing cycle \(pc+qd\) is invariant under the monodromy. 
	\end{proof}

	\begin{cor}
	\label{cor:monodromy-syz}
	The monodromies acting on \(\mathrm{H}_{1}(E_{0},\mathbb{Z})\) 
	with respect to the basis \(\{c,d\}\) around \(A\), \(B\) and \(C\)
	are given by 
	\begin{equation}
	\begin{bmatrix} -1 & -1\\ 4 & 3 \end{bmatrix},
	\begin{bmatrix} 2 & -1\\ 1 & 0 \end{bmatrix}~\mbox{and}~
	\begin{bmatrix} -1 & 4\\ 1 & 3 \end{bmatrix}.
	\end{equation}
	\end{cor}
	\begin{proof}
	Recall that from Lemma \ref{lemma:orientation-compatible-z3-action} 
	the vanishing cycles associated with the degenerate fibers \(A\),
	\(B\) and \(C\) are given by \(c-2d\), \(c+d\) and \(-2c+d\).
	The corollary directly follows from Lemma \ref{PL}.
	\end{proof}

To write down the affine structure, one needs to 
introduce one branch cut from each of the three singularities to infinity. 
Recall that \(B_{0}=B_{\mathrm{SYZ}}\setminus \{A,B,C\}\).
Let us fix a reference point \(u\in B_{0}\). 
Given a loop on the base, the affine monodromies act on \(T_{u}B_{0}\) while the monodromies of the fibration act on $\mathrm{H}_{1}(E_{u},\mathbb{Z})\cong T^{\ast}_{u} B_{0}$.  Therefore, the \emph{clockwise} affine monodromy action is dual to the \emph{counter-clockwise} monodromy of the fibration. In particular, the matrix representation of the counter-clockwise monodromy around $A$, $B$ and $C$ with respect to the dual basis $\{c,d\}$ are given by  
\begin{equation}
\label{equation:monodromy-matrices-PL-trans}
\begin{bmatrix} -1 & -1\\ 4 & 3 \end{bmatrix}^{\intercal},
\begin{bmatrix} 2 & -1\\ 1 & 0 \end{bmatrix}^{\intercal}~\mbox{and}~
\begin{bmatrix} -1 & 4\\ 1 & 3 \end{bmatrix}^{\intercal}
\end{equation}
from Corollary \ref{cor:monodromy-syz} .

To compare to \(B_{\mathrm{CPS}}\), we have further 
requirements on the branch cuts.

	\begin{lem}\label{Affine}
		There exists an affine ray emanating from each of the three singularities such that its tangent is in the monodromy invariant direction at infinity. 
     \end{lem} 	
\begin{proof} We will explain the cut emanating from $A$ and the other two are similar. 
   From Lemma \ref{Z_2 symmetry} and Corollary \ref{cor:linear-transform-complex-conjugation}, 
   the set $\{3<q<\infty\}$ is an affine line defined by \(-c=(V_{2}-V_{1})\slash 3\) and  the 
   cycle invariant under \(\phi_{0}\), which is the vanishing cycle at the 
   infinity by Lemma \ref{AKO} and Lemma \ref{lemma:orientation-compatible-z3-action}. 
\end{proof}

\begin{proof}[Proof of Theorem \ref{main thm}]
To match the affine structure on \(B_{\mathrm{SYZ}}\) with $B_{\mathrm{CPS}}$, 
we will take the branch cuts to be the affine rays in Lemma \ref{Affine}.
Notice that the orientations of vanishing 
cycles \(V_{0},V_{1},V_{2}\) are chosen as in Lemma \ref{AKO}
and Lemma \ref{lemma:orientation-compatible-z3-action}
so that they respect the \(\mathbb{Z}_{3}\)-symmetry on the \(q\)-plane. 

Recall we have the identification $\mathrm{H}_{1}(E_0,\mathbb{Z})\cong 
\mathbb{Z}c\oplus\mathbb{Z}d$ from 
Lemma \ref{lemma:orientation-compatible-z3-action}. We can 
identify $A,B,C \in B_{\mathrm{SYZ}}$ 
with $A',B',C'\in B_{\mathrm{CPS}}$ and
$v_1,v_2,v_3\in B$ with $v_1',v_2',v_3'$. There is an induced affine isomorphism 
carrying the affine triangle 
$v_1v_2v_3$ in $B_{\mathrm{SYZ}}$ to $v_1'v_2'v_3'$ in $B_{\mathrm{CPS}}$. 
Since the affine transformation acrossing the cut in $B_{\mathrm{SYZ}}$ and $B_{\mathrm{CPS}}$ 
are the same from \eqref{equation:monodromy-matrices-CPS}
and \eqref{equation:monodromy-matrices-PL-trans}, 
the affine isomorphisms glue to an affine isomorphism
$B_{\mathrm{SYZ}}\cong B_{\mathrm{CPS}}$. 
\end{proof}

\section{Floer-theoretical gluing construction of mirror geometry}\label{section: Floer mirror}

In the previous section, we have well understood the affine structure associated to the special Lagrangian fibration on $\bP^2\setminus E$, where $E$ is a smooth elliptic curve.  In this section, we construct the Floer theoretical mirror of $\bP^2$ relative to $E$, which is a direct application of the gluing method developed in \cites{CHL-glue,HKL}.

The strategy is the following.  The special Lagrangian fibration has exactly three singular fibers.  Each of these is a nodal torus pinched at one point.  However, these singular fibers are located in different energy levels, in the sense that the pseudo-isomorphisms between their formal deformations involve Novikov parameter.  The resulting mirror would be defined over $\Lambda$.

To simplify the situation, we take the following Lagrangians instead of the special Lagrangian fibers.  We take a monotone moment-map fiber of $\bP^2$, and use symplectic reduction by $\bS^1$ to construct three monotone immersed Lagrangians, which play the role of the above three singular fibers.  We consider the weakly unobstructed deformation spaces of these Lagrangians, and glue them together via quasi-isomorphisms in the Fukaya category.

Using these monotone Lagrangians, the gluing relations will be defined over $\C$, and hence we can reduce to a $\C$-valued mirror.  Moreover, the construction of \cite{CHL-glue} produces a mirror functor from the Fukaya category to the mirror matrix factorization category $\Fuk(\bP^2)\to \MF(\check{X}_\C,\tilde{W})$, which induces a derived equivalence \cite{CHL-toric}.

\subsection{The Lagrangian objects}
Let $\bL_0$ be the monotone moment-map torus fiber of $\bP^2$ equipped with the toric K\"ahler form, whose fan is generated by $e_1,e_2$ and $e_3=-e_1-e_2$, where $\{e_1,e_2\} \subset \mathfrak{t} \cong \pi_1(\bL_0)$ is the standard basis.  Consider flat connections on $\bL_0$, whose holonomies along the loops $e_1,e_2 \in \pi_1(\bL_0)$ are given by $z_1,z_2$ respectively.  Let $z_3 = 1/z_1z_2$ which is the holonomy along $e_3$.  Denote these flat connections by $\nabla^{(z_1,z_2)}$.  

The flat connections over a Lagrangian are taken over $\Lambda_0$, with holonomies $z_i \in \Lambda_0^\times$, where
$$ \Lambda_0 := \left\{\sum_{i=0}^\infty a_i \bT^{A_i} \mid a_i \in \C, A_i \geq 0 \textrm{ and increases to } +\infty \right\}$$
is the Novikov ring, and
$$ \Lambda_0^\times := \left\{\sum_{i=0}^\infty a_i \bT^{A_i} \in \Lambda_0\mid A_0 = 0 \textrm{ and } a_0\not=0 \right\} $$
is the group of invertible elements.  This ensures the Floer theory for the Lagrangian decorated by a flat connection is convergent over $\Lambda$.

Following \cite{A}, we can `push in' one of the corners of the moment map polytope.  Namely, let $\C^2_{(i)}$ be the standard coordinate charts and $X^{(i)},Y^{(i)}$ the corresponding inhomogeneous coordinates for $i=0,1,2$.  Denote the $T^2$-moment map by $$\mu:\bP^2\to \mathfrak{t}^* \textrm{ with } \mu^{-1}(\{0\})=\bL_0.$$
Here the toric K\"ahler form is taken such that the moment map image is the triangle with vertices $(-1,-1),(-1,2),(2,-1)$ (in the basis $\{e_1^\vee,e_2^\vee\} \subset \mathfrak{t}^*$).

Consider the $\bS^1$-action in each direction $e_{i+2}-e_{i+1}$ (where the subscript is mod 3).  The corresponding moment map is $(\mu,e_{i+2}-e_{i+1})$.  Moreover, the function $X^{(i)}\cdot Y^{(i)}$ is invariant under this $\bS^1$-action and gives a complex coordinate $\zeta=\zeta^{(i)}$ on the reduced space $\C^2_{(i)} \sslash \bS^1_{e_{i+2}-e_{i+1}}$.  Using this symplectic reduction, one obtains the following Lagrangian torus fibration.

\begin{prop}[\cites{Gross-eg,Goldstein}]
	For any $c\in \C$, 
	$((\mu,e_{i+2}-e_{i+1}),|X^{(i)}\cdot Y^{(i)}-c|)$ defines a Lagrangian fibration on $\C^2_{(i)}=\bP^2 - D^T_i$
	where $D^T_i$ is the toric divisor corresponding to $e_i$.
\end{prop}

When $c=0$, this is just isomorphic to the Lagrangian fibration given by the moment map.

We shall take the following Lagrangian objects.
In the reduced space $\C^2_{(i)} \sslash \bS^1_{e_{i+2}-e_{i+1}}$, $\bL_0$ is given by a circle of radius $r>0$ centered at $\zeta=0$.  Moreover  $(\mu,e_{i+2}-e_{i+1})=0$ for $\bL_0$.  For each $i=1,2,3$, we take 
$$\bL_i := \{|X^{(i)}\cdot Y^{(i)}- r| = r, (\mu,e_{i+2}-e_{i+1})=0\} \subset \bP^2$$
which is the singular fiber of the above Lagrangian fibration (for $c=r>0$).  $\bL_i$ is an immersed two-sphere with a single nodal point. We denote the immersion by $\iota_i:\bS^2\to \bP^2$ whose image is $\bL_i$.

For each $i=1,2,3$, we also have the Chekanov torus
$$\bL_i' := \{|X^{(i)}\cdot Y^{(i)}- (3r/2)| = r, (\mu,e_{i+2}-e_{i+1})=0\} \subset \bP^2.$$ 
See {\sc Figure \ref{fig:Tori-imm}}.

\begin{figure}[h]
	\begin{center}
		\includegraphics[scale=0.7]{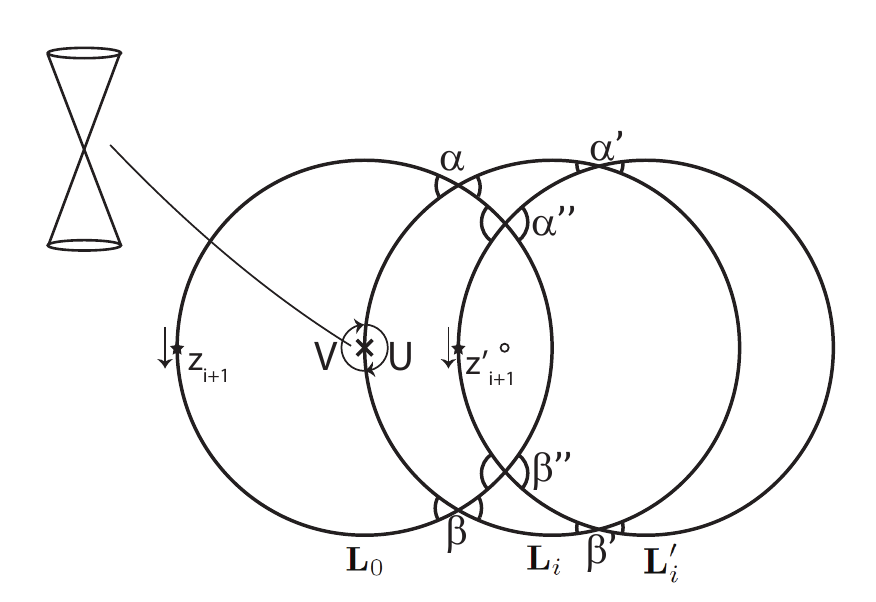}
		\caption{The images of the Lagrangians in the reduced space.}\label{fig:Tori-imm}
	\end{center}
\end{figure}

\begin{prop}
	For $t$ sufficiently small, the Lagrangians $\bL_i$ and $\bL_i'$ lie in $\bP^2 - E_t$ where $E_t=\{xyz + t(x^3 + y^3 + z^3) = 0\} \subset \bP^2$. 
\end{prop}

\begin{proof}
	$E_t$ lies in a neighborhood of the union of toric divisors $xyz=0$.  After intersecting with the moment-map level set $\{(\mu,e_{i+2}-e_{i+1})=0\}$, it is a compact set whose image in $\C\cup \{\infty\} \cong \bP^1$ under $X^{(i)}\cdot Y^{(i)}$ consists of two connected components, one is a compact simply connected region near $0$ (but does not contain $0$), and one is a compact neighborhood of $\infty$.  For $t$ small, these two regions are disjoint from the base circles of $\bL_i$ and $\bL_i'$.
\end{proof}

As explained above, we have the flat connections $\nabla^{(z_1,z_2)}$ on $\bL_0$.  Now we parametrize the flat connections on the Chekanov tori $\bL_i'$ by fixing the following trivialization of the conic fibrations.  

The conic fibration of $X^{(i)}\cdot Y^{(i)}$ restricted to $\C^2_{(i)} - \{Y^{(i)} = 0\}$ is trivial, and $Y^{(i)} \in \C^\times$ serves as the fiber coordinate.  The map 
$$((X^{(i)}Y^{(i)} - 3r/4)/|X^{(i)}Y^{(i)} - 3r/4|,Y^{(i)}/|Y^{(i)}|)$$
gives an identification of $\bL_0$ and $\bL_i'$ with $T^2$.  Thus $e_1,e_2,e_3 \in \pi_1(\bL_0)$ can be identified as elements in $\pi_1(\bL_i')$.  

Let's denote the holonomy of a flat connection over $\bL_i'$ along $e_{i+1}$ by $z_{i+1}'$, and that along the monodromy invariant direction $(e_{i+2}-e_{i+1})$ by $w_{i+1}'$.  We shall consider the objects $\left(\bL_i',\nabla^{(z_{i+1}',w_{i+1}')}\right)$.
For $\bL_0$, the holonomy of a flat connection along $(e_{i+2}-e_{i+1})$ is denoted by $w_{i+1}$, which equals to $z_{i+2} z_{i+1}^{-1}$.  

In conclusion, we shall consider the objects $(\bL_0,\nabla^{(z_1,z_2)}), \left(\bL_i',\nabla^{(z_{i+1}',w_{i+1}')}\right)$ , and the Lagrangian immersions $\bL_i$ for $i=1,2,3$.

\subsection{The Floer theoretical mirror}
We construct a mirror out of the objects $\bL_0$ and $\bL_i$.  This gives a nice application of the gluing method in \cites{CHL,HKL}.

We take a Morse model for the Lagrangian Floer theory.  Pearl trajectories, which are formed by holomorphic discs components together with gradient flow lines of a fixed Morse function, were developed in \cites{Oh,BC} for the deformation theory of monotone Lagrangians.  In \cites{FOOO-can}, the Morse model was developed to general situations using a homotopy between the Morse complex and the singular chain complex.  There is also a slightly different formulation in \cite{CW}. Such a Morse model was further developed to apply to a $G$-equivariant setting in \cites{LZ,HKLZ}.  Fixing the choice of a Morse function $f$ on a Lagrangian $L$ and perturbation datum for the pearl trajectories, an $A_\infty$ structure $\{m_k: k\in \Z_{\geq 0}\}$ is constructed on the space of chains $\CF(L)$ generated by critical points of $f$.  Moreover, given a degree-one chain $b\in \CF^1(L)$, one has the deformed $A_\infty$ structure $\{m_k^b: k\in \Z_{\geq 0}\}$ \cites{FOOO-T1}.  $L$ can also be decorated by flat connections $\nabla$, which produce $\{m_k^{(L,\nabla)}: k\in \Z_{\geq 0}\}$.

The holomorphic discs bounded by the torus $\bL_0 \subset \bP^2$ were known due to the classification by \cite{CO}.  Moreover, $(\bL_0, \nabla^{(z_1,z_2)})$ are weakly unobstructed \cite{FOOO-T1}, namely,
$$ m_0^{(\bL_0, \nabla^{(z_1,z_2)})} = W \cdot \one_{\bL_0} $$
where $\one_{\bL_0}$ is the unit.  
The disc potential is given by
$$ W = T^{A/3} \left( z_1 + z_2 + \frac{1}{z_1z_2}\right)= T^{A/3} \left( z_1 + z_2 +z_3 \right) $$
where $A$ is the area of the line class in $\bP^2$.

For the grading of the Lagrangians, 
for each $i=1,2,3$, we consider the anti-canonical divisor 
$$D_i := \{X^{(i)}Y^{(i)}= 3r/4\} \cup \{z^{(i)} = 0\}$$
(where $z^{(i)}$ is the homogeneous coordinate that defines the toric divisor $D_i^T=\{z^{(i)} = 0\}$). 

\begin{lem} \label{lem:graded}
	$\bL_0$, $\bL_i$ and $\bL_i'$ are graded Lagrangians in the complement $\bP^2 - D_i$.
\end{lem}
\begin{proof}
	$\bL_0$, $\bL_i$ and $\bL_i'$ are isotopic to special Lagrangian fibers with respect to the holomorphic volume form $dX^{(i)}\wedge dY^{(i)} / (X^{(i)}Y^{(i)} - 3r/4)$ defined on $\bP^2 - D_i$, and hence they are graded.
\end{proof}

Then the Maslov index formula of \cites{CO,A} can be applied and one has the following.

\begin{prop}[\cites{CO,A}]
	The Maslov index of a disc $\beta$ bounded by  $\bL_0, \bL_i, \bL_i'$ equals to 
	$\mu(\beta) = 2 \, \beta \cdot D_i.$
\end{prop}

Now we fix a choice of Morse functions on the Lagrangians.
In above we have fixed an identification of $\bL_0$ and $\bL_i'$ with the standard $T^2$.  Let's take a perfect Morse function on $T^2$ such that the unstable circles of the two degree-one critical points are dual to the $\bS^1$-orbits in the directions of $e_1$ and $e_2$ respectively.  By abuse of notation, we also denote these two degree-one critical points by $e_1,e_2$. The maximum and minimum points are denoted by $\one$ and $e_{12}$ respectively.

For the immersed Lagrangians $\bL_i$, the choice of Morse functions is more subtle and we proceed as follows. First, consider the immersed generators for the Floer theory.  The domain of the immersion is $\bS^2$.  The inverse image of the transverse self-nodal point consists of two points $q_1,q_2 \in \bS^2$.  The branch jumps $q_1\to q_2$ and $q_2\to q_1$ are denoted by $U_i$ and $V_i$ respectively.  See {\sc Figure \ref{fig:Tori-imm}}.  By using the grading in Lemma \ref{lem:graded}, it is easy to see the following.

\begin{lem}
	Both $U_i, V_i \in \mathrm{CF}(\bL_i,\bL_i)$ have $\deg =1$.  
\end{lem}

We use $U_i,V_i$ for the Maurer-Cartan deformations of $\bL_i$.  By using a $\Z_2$-symmetry, they can be shown to be unobstructed:

\begin{lem}[\cite{HKL}*{Lemma 3.3}]
$u_iU_i + v_iV_i\in \mathrm{CF}(\bL_i,\bL_i)$ are bounding cochains for $\bL_i \subset \bP^2 - D_i$, namely, $m_0^{u_iU_i + v_iV_i}=0$,
where
$$(u_i,v_i) \in \Lambda_0^2 - \{\val (u_iv_i) = 0\}.$$
\end{lem}

It is important to take $\val (u_i v_i) > 0$, since there are constant polygons at the nodal point (whose number of $U_i$ corners must equal to the number of $V_i$ corners to go back to the same branch) contributing to the Floer theory of $\bL_i$.  This ensures Novikov convergence of $m_0^{u_iU_i + v_iV_i}$.

We construct isomorphisms between $(\bL_0,\nabla^{(z_1,z_2)})$ and 
$(\bL_i,u_iU_i+v_iV_i)$ under suitable gluing relations between $(z_1,z_2)$ and $(u_i,v_i)$.  Observe that
$\bL_i$ intersects cleanly with $\bL_0$ (or $\bL_i'$) at two circle fibers $(\alpha,\beta)$ (or $(\alpha',\beta')$) over the two intersection points of the base loci $|\zeta|=r$ and $|\zeta-r|=r$ (or $|\zeta-3r/2| = r$) in the $\zeta$-plane.  Similarly, $\bL_0$ intersects with $\bL_i'$ at two circles $(\alpha'',\beta'')$.  We fix a perfect Morse function on each of these circles.  The maximum and minimum points are denoted by $\alpha\otimes \one, \alpha\otimes \mathbf{m}$ respectively (and similar for $\beta,\alpha',\beta',\alpha'',\beta''$, where $\mathbf{m}$ stands for `minimum'). 

$$\mathrm{CF}(\bL_i,\bL_0) = \mathrm{Span}_\Lambda \{\alpha\otimes \one, \alpha\otimes \mathbf{m}, \beta\otimes \one, \beta\otimes \mathbf{m}\}$$ 
which have degrees $0,1,1,2$ respectively.  We can also regard them as generators of $\mathrm{CF}(\bL_0,\bL_i)$, and they have degrees $1,2,0,1$ respectively.

By the projection to the complex $\zeta$-plane, one can deduce the following (see \cite{HKL}*{Section 3.3}), which is important for computing $m_1^{\bL_i,\bL_0}(\alpha\otimes\one)$ and $m_1^{\bL_i',\bL_i}(\alpha'\otimes\one)$.

\begin{lem}
	In $\bP^2 - D_i$, $\bL_i$ and $\bL_0$ (or similarly $\bL_i$ and $\bL_i'$) bound exactly two non-constant Maslov-two holomorphic polygons that have output to $\beta\otimes\one$ (or $\beta' \otimes \one$).  One of them has corners at $\alpha, \beta$ (or $\alpha',\beta'$).  The other has corners at $\alpha, \beta, V$ (or $\alpha',\beta',U$). 
\end{lem}

The Morse function on $\bL_i$ that we choose is the following.
The boundaries of the above two holomorphic polygons in $\bL_i$ give two curved segments.  We take a perfect Morse function on the domain $\bS^2$ of $\bL_i$ such that the two critical points lie in $\bS^2-\{q_1,q_2\}$, and the two flow lines connecting $q_1,q_2$ to the minimum are distinct and do not intersect with any of these curve segments.

Then we have the following isomorphisms between the Lagrangian branes.

\begin{thm}
	$\alpha\otimes\one \in \mathrm{CF}((\bL_0,\nabla^{(z_1,z_2)}),(\bL_i,u_iU_i+v_iV_i))$ is an isomorphism if and only if $v_i = z_{i+1}^{-1}$ and $u_iv_i = 1 + z_{i}^{-1}z_{i+1}^{-2}$
	where the subscripts are mod 3.
\end{thm}
\begin{proof}
	Fix $i=1,2,3$.  First we consider $\alpha'' \otimes\one \in \mathrm{CF}((\bL_0,\nabla^{(z_{i+1},w_{i+1})}),(\bL_i',\nabla^{(z_{i+1}',w_{i+1}')}))$ between the tori.  		$m_{1,u}(\alpha'' \otimes\one)$ has degree $\deg(\alpha'' \otimes\one) +1 - \mu(u) \geq 0$ where $\mu(u)$ is the Chern-Weil Maslov index of the strip class $u$.  Since  $\alpha'' \otimes\one$ has degree zero and the minimal Maslov index for $\bL_0$ and $\bL_i$ is zero,
	$m_1(\alpha'' \otimes\one)$ is merely contributed by strips with Chern-Weil Maslov index zero.  We have $\mu(u)=2 u \cdot D_i$.  Thus any $u$ which contributes to $m_1(\alpha'' \otimes\one)$ does not intersect with $D_i$.   We have
	$$ m_1(\alpha'' \otimes\one) = (1-w_{i+1}'w_{i+1}^{-1}) \alpha'' \otimes \mathbf{m} + (1+ w_{i+1} - z_{i+1}'z_{i+1}^{-1}s) \beta'' \otimes \one$$
	where the first term is contributed by the two flow lines from $\alpha'' \otimes\one$ to $\alpha'' \otimes \mathbf{m}$, and the second term is contributed from the holomorphic strips from $\alpha'' \otimes\one$ to $\beta'' \otimes \one$ \cites{Seidel, PT, HKL}.  Hence the cocycle condition $m_1(\alpha'' \otimes\one)=0$ implies $w_{i+1}'=w_{i+1}=z_{i+1}^{-1}z_{i+2}=z_i^{-1}z_{i+1}^{-2}$ and $z_{i+1}' = z_{i+1}(1+w_{i+1})$.  Moreover, the strips also give $m_2(\beta''\times \one,\alpha''\times \one) = \one_{\bL_0}$ and $m_2(\alpha''\times \one,\beta''\times \one) = \one_{\bL_i'}$.  Thus $\alpha''\times \one$ is an isomorphism if and only if the above relations hold.
	
	Now we consider $m_1(\alpha\otimes \one)$ and $m_1(\alpha'\otimes\one)$.  We have 
	\begin{align*}
	m_1(\alpha\otimes \one) =& (u_i-z_{i+1}') \beta \otimes \one + (w_{i+1} - f(u_iv_i)) \alpha \otimes \mathbf{m} \\
	m_1(\alpha'\otimes\one)=& (v_i-z_{i+1}^{-1}) \beta' \otimes \one + (w_{i+1}' - g(u_iv_i)) \alpha \otimes \mathbf{m}\\
	\end{align*} 
	for some series $f$ and $g$.  Requiring them to be zero implies $u_i=z_{i+1}'$, $v_i=z_{i+1}^{-1}$, $w_{i+1}=f(u_iv_i), w_{i+1}'=g(u_iv_i)$.  It easily follows that 	$\alpha$ and $\alpha'$ are isomorphisms under the above relations.
	
	Since $m_2(\alpha,\alpha')=\alpha''$, $\alpha'' \otimes \one$ is also an isomorphism under the above relations.  Thus $w_{i+1}'=w_{i+1}=z_i^{-1}z_{i+1}^{-2}$ and $z_{i+1}' = z_{i+1}(1+w_{i+1})$, implying $f(u_iv_i)=g(u_iv_i)$ and $u_iv_i=1+w_{i+1}$.  Result follows.
\end{proof}

According to the above theorem, the formal deformation spaces of $\bL_0$ and $\bL_i$ for $i=1,2,3$ are glued by the transitions  $v_i = z_{i+1}^{-1}$ and $u_iv_i = 1 + z_{i}z_{i+1}^2$.  We denote the resulting space by $\check{X}$.  It consists of the chart $(\Lambda_0^\times)^2$ coming from the torus $\bL_0$, and the charts $(\Lambda_0^2)_{(i)} - \{\val(u_iv_i)=0\}$ coming from the immersed sphere $\bL_i$ for $i=1,2,3$.  

$\check{X}$ is defined over $\Lambda$.  On the other hand, note that the transition functions do not involve the Novikov parameter $\bT$.  This is because the base circles of $\bL_0$, $\bL_i$ and $\bL_i'$ in the reduced space are taken to be the same size, so that the symplectic areas of strips are the same.
The $\C$-valued part of $\check{X}$ is denoted by $\check{X}_\C$, which is the union of the $\C$-valued parts of the charts of $\check{X}$.  

\begin{rmk}
	The $\C$-valued part of the chart $(\Lambda_0^2)_{(i)} - \{\val(u_iv_i)=0\}$ of the immersed Lagrangian $\bL_i$ is the singular conic
	$$\{(u_i,v_i)\in \C^2: u_iv_i=0\} = \{(u_i,0): u_i \in \C^\times\} \cup \{(0,v_i): v_i \in \C^\times\} \cup \{(0,0)\}$$
	whose valuation is $\{(0,+\infty)\} \cup \{(+\infty,0)\} \cup \{(+\infty,+\infty)\}$.  Note that this subset is disconnected under the non-Archimedian topology.  Moreover, the $\C$-valued part of the gluing region with the torus chart $(\C^\times)^2 \subset (\Lambda_0^\times)^2$ is $\{(0,v_i): v_i \in \C^\times\}$.  This is not of the correct complex dimension.  Thus we first work over $\Lambda$ to construct the mirror, and then we can restrict to $\C$ to get the $\C$-valued mirror.
\end{rmk}

\begin{rmk} \label{rmk:C}
	In the above Floer theoretical construction, the mirror is simply glued from one torus chart $(\C^\times)^2$ and three charts coming from immersed spheres.  On the other hand, the corresponding cluster variety consists of infinitely many torus charts.
\end{rmk}

\subsection{Identification with the Carl--Pomperla--Siebert mirror}

Now we show that the resulting geometry from the above construction agrees with the Carl--Pomperla--Siebert mirror.  This gives Theorem \ref{thm:agree}.

\begin{prop}
	$\check{X}_\C$ is the blowing up at three points in the three toric divisors of the toric variety whose fan has the rays generated by $(2,-1),(-1,2)$ and $(-1,-1)$, with the strict transform of the toric divisors removed.
	$W_\C = z_1 + z_2 + \frac{1}{z_1z_2}$ on $(\C^\times)^2$ extends to be a proper elliptic fibration on $\check{X}_\C$ with three \(I_{1}\)-fibers.
\end{prop}
\begin{proof}
	The blowing up of the toric chart $\C_{(V)} \times \C^\times_{(Z)}$ at $(V,Z)=(0,-1)$ has local charts $\C^2_{(U,V)}-\{UV=1\}$ and $\C^2_{(\tilde{V},\tilde{Z})} - \{\tilde{Z}=1\}$ with the change of coordinates $V=\tilde{V}\tilde{Z}$ and $U=\tilde{V}^{-1}$ (where $\tilde{Z}=Z+1$).  The strict transform of the toric divisor $\{V=0\} \subset \C_{(V)} \times \C^\times_{(Z)}$ is given by $\tilde{V}=0$, and its complement in the blowing-up is identified with the chart $u_iv_i = 1 + z_{i}^{-1}z_{i+1}^{-2}$ of $\check{X}_\C$
	via $\tilde{Z}=1 + z_{i}^{-1}z_{i+1}^{-2}$, $U=u_i$, $V=v_i$.  The open torus orbit $\C^\times_{(V)} \times \C^\times_{(Z)}$ is identified with the torus chart of $\bL_0$ by $V=z_{i+1}^{-1}$ and $Z=z_{i}^{-1}z_{i+1}^{-2}$.  This gives the identification between $\check{X}_\C$ and the blowing-up.
	
	We already know that $W$ on $(\C^\times)^2_{(z_1,z_2)}$ gives a fibration whose generic fibers are three-punctured elliptic curves.  $W$ has three critical values, whose fibers are 3-punctured $A_1$ singular fibers.  Below, we see that the partial compactification by the immersed charts $\C^2_{(i)} - \{u_iv_i=1\}$ exactly fill in the punctures in all elliptic fibers.
	
	Consider a fiber $W=c$ for $c \in \C$.  
	For the chart $\C^2_{(i)} - \{u_iv_i=1\}$, $u_iv_i = 1 + z_{i+1}^{-1}z_{i+2}$.  Thus $z_{i+2}=(u_iv_i-1)z_{i+1}=(u_iv_i-1) v_i^{-1} = u_i - v_i^{-1}$.  Then
	$$ W = z_{i+1} + z_{i+2} + \frac{1}{z_{i+1}z_{i+2}} = u_i + v_i^2(u_iv_i-1)^{-1} $$
	in the chart.  The fiber is given by
	$$ u_i(u_iv_i-1) + v_i^2 = c(u_iv_i-1). $$
	The partial compactification coming from this chart is $v_i=0$.  Thus it adds the point $(u_i,v_i)=(c,0)$ to the fiber.  In other words, the coordinate axes $v_i=0$ are sections of the fibration of $W$.  The partial compactification adds in these three sections which are exactly the union of three punctures of the elliptic fibers.
\end{proof}

We note that the meromorphic functions $u_i$ for $i=1,2,3$ satisfy the following explicit equation.

\begin{prop}
$$ (u_1^3+u_2^3+u_3^3) + 2u_1u_2u_3 - \sum_{i=1}^3 (u_1^2 u_2 + u_1 u_2^2) =0.$$
\end{prop}

\begin{proof}
We have 
\begin{equation}
z_0z_1z_2=1.
\label{eq:1}
\end{equation}
Moreover,
\begin{equation}
u_i = z_{i+1}(1 + z_{i+1}^{-1}z_{i+2})=z_{i+1}+z_{i+2}.
\label{eq:2}
\end{equation}
We compute $u_i^3$, $u_i^2 u_{i+1}, u_iu_{i+1}^2$ and $u_1u_2u_3$ using \eqref{eq:2}.  It turns out the variables $z_0,z_1,z_2$ can all be eliminated and we obtain the resulting equation.

\end{proof}

%

\appendix
\section{The proof of Proposition \ref{proposition:orientation}}
\label{appendix:proof-of-proposition}
We resume the notation introduced in \S\ref{section: proof}.
Abusing the notation, for \(q\in \{r\zeta^{j}:r\le 3\}\), let 
\(\gamma_{j}\) be the line segment connecting \(q\) and \(3\zeta^{j}\) and
\(\Gamma_{j}^{\gamma_{j}}(q)\) to denote the set-theoretic union
\begin{equation}
\bigcup_{q'\in\gamma_{j}} V_{j}^{(q')}.
\end{equation}
We also denote by \(\vec{\gamma}_{j}\) the line segment \(\gamma_{j}\) 
equipped with an orientation from \(3\zeta^{j}\) towards to \(q\)
and by \(\vec{\Gamma}_{j}^{\gamma_{j}}(q)\) the set \(\Gamma_{j}^{\gamma_{j}}(q)\)
with the induced orientation as in .
The integral
\begin{equation}
\int_{\vec{\Gamma}_{j}^{\gamma_{j}}(q)}\check{\Omega}
\end{equation}
becomes a function in \(q\in W_{j}\). 
For simplicity, we put
\begin{equation}
\label{equation:definition-of-g(q)}
G(q):=\int_{\vec{\Gamma}_{0}^{\gamma_{0}}(q)}
\check{\Omega},~\mbox{where}~q\in(-\infty,3]\subset\mathbb{C}.
\end{equation}
Proposition \ref{proposition:orientation}
is an immediate consequence of the following lemma. 
\begin{lem}
\label{lemma:calculation-lemma}
\begin{equation*}
\int_{\vec{\Gamma}_{0}^{\gamma_{0}}(q)} \check{\Omega} \in\sqrt{-1}
\mathbb{R}_{+}~\mbox{for}~q\in(-\infty,3].
\end{equation*}
\end{lem}
\begin{proof}
Using \(\check{\Omega}\) is \(\mathrm{d}\)-closed and independent of \(q\),
we compute
\begin{align*}
\begin{split}
\frac{\mathrm{d}G(q)}{\mathrm{d}q} &= \int_{\partial\vec{\Gamma}_{0}^{\gamma_{0}}(q)} 
\iota_{\frac{\partial}{\partial q}}\check{\Omega}.
\end{split}
\end{align*}
From the construction, \(\partial \vec{\Gamma}_{0}^{\gamma_{0}}(q)\) is equal to
\(V_{0}^{(q)}\) as 
\emph{oriented cycles}. 

Recall that \(E_{q}\cap(\mathbb{C}^{\ast})^{2}:=
\left\{(t_{1},t_{2})\in(\mathbb{C}^{\ast})^{2}\colon t_{1}+t_{2}+(t_{1}t_{2})^{-1}=q\right\}\).
Let \(\delta_{0}(q)\) be the image of \(V_{0}^{(q)}\) under the projection 
\begin{equation}
(\mathbb{C}^{\ast})^{2}\to\mathbb{C}^{\ast},~(t_{1},t_{2})\mapsto t_{1}.
\end{equation}
For \(q\in (-\infty,3]\), \(E_{q}\cap(\mathbb{C}^{\ast})^{2}\to\mathbb{C}^{\ast}\)
admits three ramifications: only one of them lies on the real axis and
the other two are symmetric with respect to the real axis, denoted by \(x\) and \(\bar{x}\).
Here we assume that \(\mathrm{Im}(x)>0\).
\(x\) and \(\bar{x}\) are connected through \(\delta_{0}(q)\).
We equip \(\delta_{0}(q)\) with an orientation going from \(x\) to \(\bar{x}\).
Note that \(\delta_{0}(0)\equiv \delta_{0}\) as oriented cycles.

We can write \(\partial\vec{\Gamma}_{0}^{\gamma_{0}}(q)=V_{0}^{(q)}=\partial\Gamma^{+}(q)\cup 
\partial\Gamma^{-}(q)\), union of the graph of \(t_{2}^{+}\) and
the graph of \(t_{2}^{-}\) along \(\delta_{0}(q)\) as in the paragraph \textbf{(C)}. 
Then
\begin{align}
\label{equation:integral-to-compute}
\frac{\mathrm{d}G(q)}{\mathrm{d}q} &= \int_{\partial\vec{\Gamma}_{0}^{\gamma_{0}}(q)}
\iota_{\frac{\partial}{\partial q}}\check{\Omega}\notag\\
&=\int_{\partial\Gamma^{+}(q)}
\iota_{\frac{\partial}{\partial q}}\check{\Omega}+\int_{\partial\Gamma^{-}(q)}
\iota_{\frac{\partial}{\partial q}}\check{\Omega}\notag\\
&=\int_{\delta_{0}(q)}\frac{\sqrt{-1}}{\sqrt{(q-t_{1})^{2}-4/t_{1}}}
\frac{\mathrm{d}t_{1}}{t_{1}}+
\int_{-\delta_{0}(q)}
\frac{-\sqrt{-1}}{\sqrt{(q-t_{1})^{2}-4/t_{1}}}\frac{\mathrm{d}t_{1}}{t_{1}}.
\end{align}
We explain the third equality above.
Restricting on \(\partial\Gamma^{+}(q)\) or \(\partial\Gamma^{-}(q)\) and
making use of the equation \(t_{1}^{2}t_{2}+t_{1}t_{2}^{2}+1 = qt_{1}t_{2}\),
we obtain
\begin{align}
\label{equation:differential-form-q-t1-coordinate}
\begin{split}
\left.\iota_{\frac{\partial}{\partial q}}\check{\Omega}\right|_{\partial\Gamma^{\pm}(q)} 
&=\sqrt{-1}\cdot \frac{t_{2}^{\pm}}{q\cdot t_{1}t_{2}^{\pm}-t_{1}^{2}t_{2}^{\pm}-2t_{1}(t_{2}^{\pm})^{2}}
\mathrm{d}t_{1}\\
&=\sqrt{-1}\cdot\frac{1}{q-t_{1}-2t_{2}^{\pm}}\frac{\mathrm{d}t_{1}}{t_{1}}.
\end{split}
\end{align}
Also from \eqref{equation:t2-equation-solve}, we see that 
\begin{equation}
\label{equation:solve-t2}
2t_{2}^{\pm}+t_{1}-q = \pm \sqrt{(q-t_{1})^{2}-4/t_{1}}.
\end{equation}
Together with the induced orientation on \(\partial\Gamma^{\pm}(q)\),
we arrive at the desired equality.
Note that the branched cut of \(\sqrt{z}\) in \eqref{equation:solve-t2}
is chosen such that
\begin{equation*}
\displaystyle\sqrt{z}=\exp\left(\frac{1}{2}{\log z}\right),~
\log{z}=r+\sqrt{-1}\theta~\mbox{with}~\theta\in [0,2\pi).
\end{equation*}

Since both of the integrands in \eqref{equation:integral-to-compute} are holomorphic, 
we can deform the cycle \(\delta_{0}(q)\) a little bit.
We have the following two cases: \textbf{(a)} \(0<q<3\) and \textbf{(b)} \(q<0\).

For the case \textbf{(a)}, we can deform \(\delta_{0}(q)\)
into a \emph{circular} arc \(\delta_{0}'(q)\) 
joining the end points \(x\) and \(\bar{x}\)
without touching the third ramification point \(y\), where 
the integrands have a pole (cf.~\textsc{Figure} \ref{figure:deformed-contour-I}). 

\begin{figure}
  	\centering
	\includegraphics[scale=0.65]{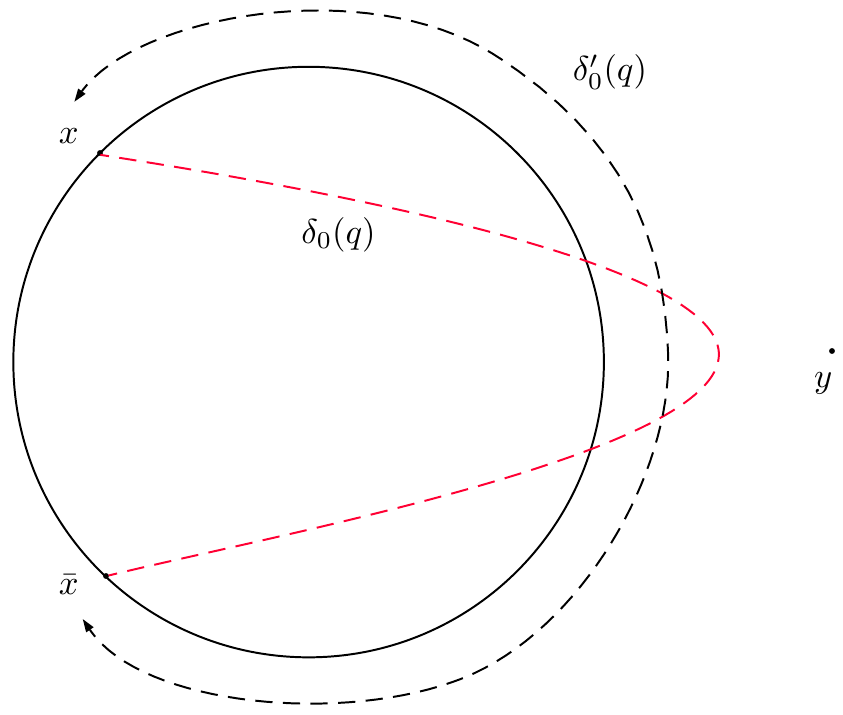}
  	\caption{The deformed contour \(\delta_{0}'(q)\).}
   	\label{figure:deformed-contour-I}
\end{figure}

Moreover, on the circular arc, we have
\begin{equation}
\label{equation:imaginary-part-of-integrands}
\mathrm{Im}~\sqrt{(q-t_{1})^{2}-4/t_{1}}\ge 0.
\end{equation}
Therefore,
\begin{equation*}
\int_{\delta_{0}'(q)}
\frac{\sqrt{-1}}{\sqrt{(q-t_{1})^{2}-4/t_{1}}}\frac{\mathrm{d}t_{1}}{t_{1}}
\end{equation*}
has negative imaginary part and so does \eqref{equation:integral-to-compute}. 
This implies that the imaginary part of \(G(q)\) decreases.

For the case \textbf{(b)}, we can deform \(\delta_{0}(q)\)
into the contour \(\delta_{0}''(q)\) (cf.~\textsc{Figure} \ref{figure:deformed-contour-II}). 
\begin{figure}
 	\centering
	\includegraphics[scale=0.7]{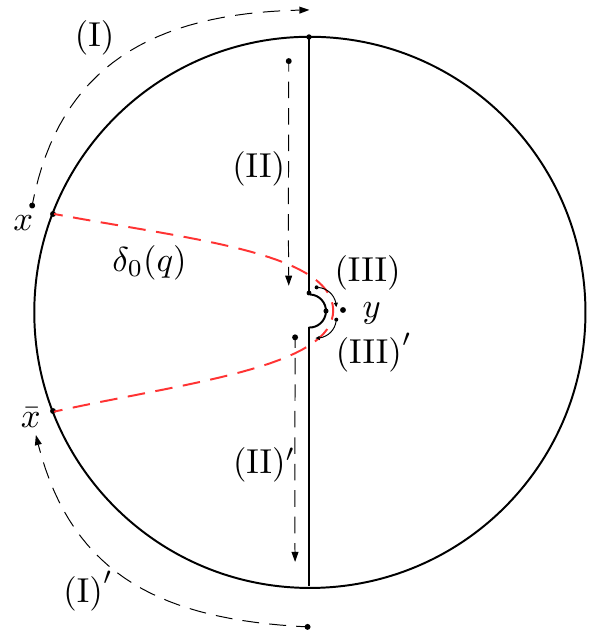}
  	\caption{The deformed contour \(\delta_{0}''(q)\),
  	which is the union of \(\mathrm{(I)}\sim\mathrm{(III)}\) and 
  	\(\mathrm{(I)}'\sim\mathrm{(III)}'\).}
   	\label{figure:deformed-contour-II}
\end{figure}
By symmetry, it suffices to compute the integral over \(\mathrm{(I)}\sim\mathrm{(III)}\).
The equation \eqref{equation:imaginary-part-of-integrands} still holds for 
\textrm{(I)} and \textrm{(III)}. On the contour \textrm{(II)},
with the parameterization \(t_{1}=\sqrt{-1}\cdot r\),
\begin{align*}
\begin{split}
(q-\sqrt{-1}r)^{2}+4\frac{\sqrt{-1}}{r}
&=q^{2}-2qr\sqrt{-1}-r^{2}+\frac{4\sqrt{-1}}{r}\\
&=(q^{2}-r^{2})+\sqrt{-1}\left(\frac{4}{r}-2qr\right)
\end{split}
\end{align*}
has positive imaginary part if \(r>0\), which guarantees that
\(\sqrt{(q-t_{1})^{2}-4/t_{1}}\) has \emph{positive real part} if \(r>0\).
Also we have \(\mathrm{d}t_{1}\slash t_{1}=\mathrm{d}r\slash r\).
These implies again that 
\begin{equation*}
\int_{\mathrm{(II)}}
\frac{\sqrt{-1}}{\sqrt{(q-t_{1})^{2}-4/t_{1}}}\frac{\mathrm{d}t_{1}}{t_{1}}
\end{equation*}
has negative imaginary part.

We deduce from above that in both cases, \eqref{equation:integral-to-compute}
has negative imaginary parts.
Together with the fact \(G(3)=0\), it follows that
\(G(q)\in \sqrt{-1}\cdot \mathbb{R}_{+}\) for \(q<3\).
\end{proof}
\begin{cor}
We have \(G(0)\in\sqrt{-1}\cdot \mathbb{R}_{+}\).
\end{cor}
\begin{proof}
This immediately follows from Lemma \ref{lemma:calculation-lemma}.
\end{proof}
\begin{cor}
\label{corollary:infinity}
\(\lim_{q\to-\infty} G(q) = \sqrt{-1}\cdot\infty\).
\end{cor}
\begin{proof}
Assume \(q<0\). 
We adapt the notation in \textsc{Figure~\ref{figure:deformed-contour-II}}.
To compute the integral \eqref{equation:integral-to-compute},
as in the proof of Lemma \ref{lemma:calculation-lemma},
we can deform the path \(\delta_{0}(q)\) to \(\delta_{0}''(q)\).
We put \(r:=|x|\).

Note that 
\begin{equation*}
\frac{-\sqrt{-1}}{\sqrt{(q-t_{1})^{2}-4/t_{1}}}\frac{\mathrm{d}t_{1}}{t_{1}}
\end{equation*}
has negative imaginary part on the whole \(\mathrm{(II)}\).
In particular, we have
\begin{equation}
\label{equation:estimate-integrals}
\mathrm{Im}\int_{\delta_{0}''(q)}
\frac{\sqrt{-1}}{\sqrt{(q-t_{1})^{2}-4/t_{1}}}\frac{\mathrm{d}t_{1}}{t_{1}}\le
\mathrm{Im}\int_{\mathcal{C}}
\frac{\sqrt{-1}}{\sqrt{(q-t_{1})^{2}-4/t_{1}}}\frac{\mathrm{d}t_{1}}{t_{1}},
\end{equation}
where \(\mathcal{C}\) is the (clockwise oriented) contour 
\begin{equation*}
re^{i\theta}~\mbox{with}~\theta\in \left[\frac{3\pi}{4},\frac{\pi}{2}\right].
\end{equation*}
It suffices to estimate the right hand side of \eqref{equation:estimate-integrals}.
On \(\mathcal{C}\), we have
\begin{equation}
(q-t_{1})^{2}-4\slash t_{1} \sim (q-t_{1})^{2}=
q^{2}\left(1-\frac{r}{q}e^{i\theta}\right)^{2}
\end{equation}
provided \(|q|\) is large enough. In the meanwhile, 
\(r/|q|\sim 1\). It is not hard to see that 
\begin{equation*}
\mathrm{Im}\left(\frac{1}{\sqrt{(q-t_{1})^{2}-4\slash t_{1}}}\right)
\ge \frac{\kappa}{|q|},
\end{equation*}
for some positive constant \(\kappa\). Since \(\mathcal{C}\)
is clockwise oriented, we have
\begin{equation}
\mathrm{Im}\int_{\delta_{0}''(q)}
\frac{\sqrt{-1}}{\sqrt{(q-t_{1})^{2}-4/t_{1}}}\frac{\mathrm{d}t_{1}}{t_{1}}\le
-\frac{\kappa\cdot\mathrm{length}(\mathcal{C})}{|q|}=:-\frac{\kappa'}{|q|}.
\end{equation}
This shows that 
\begin{equation}
\frac{\mathrm{d}(\mathrm{Im}~G(q))}{\mathrm{d}q}\le \frac{\kappa'}{q}~\mbox{for all}~q\ll 0
\end{equation}
and therefore \(\lim_{q\to-\infty} \mathrm{Im}~G(q)=\infty\).
\end{proof}


\begin{bibdiv}
\begin{biblist}

\bib{A}{article}{
  author={Auroux, Dennis},
  title={Mirror symmetry and T-duality in the complement of an anticanonical divisor},
  journal={J. Gökova Geom. Topol.},
  volume={1},
  date={2007},
  pages={51--91},
}

\bib{A2}{article}{
  author={Abouzaid, Mohammed},
  title={The family Floer functor is faithful},
  journal={J. Eur. Math. Soc. (JEMS)},
  volume={19},
  date={2017},
  number={7},
  pages={2139--2217},
}

\bib{A3}{article}{
  author={Abouzaid, Mohammed},
  title={Homological mirror symmetry without corrections},
  journal={preprint, arXiv:1703.07898},
}

\bib{AAK} {article}{
	AUTHOR = {Abouzaid, Mohammed},
	AUTHOR = {Auroux, Denis},
	AUTHOR = {Katzarkov, Ludmil},
	TITLE = {Lagrangian fibrations on blowups of toric varieties and mirror symmetry for hypersurfaces},
	JOURNAL = {Publ. Math. Inst. Hautes \'{E}tudes Sci.},
	VOLUME = {123},
	YEAR = {2016},
	PAGES = {199--282},
}

\bib{AKO}{article}{
  author={Auroux, Denis},
  author={Katzarkov, Ludmil},
  author={Orlov, Dmitri},
  title={Mirror symmetry for del Pezzo surfaces: vanishing cycles and coherent sheaves},
  journal={Invent. Math.},
  volume={166},
  date={2006},
  number={3},
  pages={537--582},
}

\bib{BC}{article}{
	author = {Biran, Paul},
	author = {Cornea, Octav},
	TITLE = {A {L}agrangian quantum homology},
	BOOKTITLE = {New perspectives and challenges in symplectic field theory},
	SERIES = {CRM Proc. Lecture Notes},
	VOLUME = {49},
	PAGES = {1--44},
	PUBLISHER = {Amer. Math. Soc., Providence, RI},
	YEAR = {2009},
}

\bib{CHL}{article}{
	author = {Cho, Cheol-Hyun},
	author = {Hong, Hansol},
	author = {Lau, Siu-Cheong},
	Journal = {J. Differential Geom.},
	Number = {1},
	Pages = {45--126},
	Title = {Localized mirror functor for {L}agrangian immersions, and homological mirror symmetry for {$\mathbb{P}^1_{a,b,c}$}},
	Volume = {106},
	Year = {2017}}

\bib{CHL-nc}{article}{
	author = {Cho, Cheol-Hyun},
	author = {Hong, Hansol},
	author = {Lau, Siu-Cheong},
	Journal = {to appear in Mem. Amer. Math. Soc.},
	Title = {Noncommutative homological mirror functor},
}

\bib{CHL-glue}{article}{
	author = {Cho, Cheol-Hyun},
	author = {Hong, Hansol},
	author = {Lau, Siu-Cheong},
	title={Gluing Localized Mirror Functors},
	journal={preprint, arXiv:1810.02045},
}

\bib{CHL-toric}{article}{
	author = {Cho, Cheol-Hyun},
	author = {Hong, Hansol},
	author = {Lau, Siu-Cheong},
	title={Localized mirror functor constructed from a {L}agrangian
		torus},
	journal = {J. Geom. Phys.},
	volume = {136},
	year = {2019},
	pages = {284--320},
}

\bib{CO}{article}{
	author = {Cho, Cheol-Hyun},
	author = {Oh, Yong-Geun},
	title = {Floer cohomology and disc instantons of {L}agrangian torus 
	fibers in {F}ano toric manifolds},
	journal = {Asian J. Math.},
	year = {2006},
	volume = {10},
	pages = {773--814},
	number = {4},
}

\bib{CJL}{article}{
  author={Collins, Tristan},
  author={Jacob, Adam},
  author={Lin, Yu-Shen},
  title={Special {L}agrangian submanifolds of log {C}alabi--{Y}au manifolds},
  journal={preprint 2019, arXiv://1904.08363},
}

\bib{CJL2}{article}{
     author={Collins, Tristan},
    author={Jacob, Adam},
    author={Lin, Yu-Shen},
    journal={in preparation},
}



\bib{CLL}{article}{
  author={Chan, Kwokwai},
  author={Lau, Siu-Cheong},
  author={Leung, Naichung Conan},
  title={SYZ mirror symmetry for toric Calabi-Yau manifolds},
  journal={J. Differential Geom.},
  volume={90},
  date={2012},
  number={2},
  pages={177--250},
}

\bib{CPS}{article}{
  author={M. Carl},
  author={M. Pumperla},
  author={Siebert, Berndt},
  title={A tropical view on Landau--Ginzburg models},
  journal={unpublished (preliminary version), 2011},
}

\bib{CW}{article}{
	Author = {Francois Charest and Chris Woodward},
	Journal = {arXiv preprint},
	Title = {Floer theory and flips},
	year={2015},
	Note = {\href{https://arxiv.org/abs/1508.01573}{arXiv:1508.01573}},
}

\bib{DGI}{article}{
  author={Dimitroglou Rizell, Georgios},
  author={Goodman, Elizabeth},
  author={Ivrii, Alexander},
  title={Lagrangian isotopy of tori in $S^2\times S^2$ and $\mathbb{C}P^2$},
  journal={preprint 2016, arXiv://1602.08821},
}

\bib{F0}{article}{
  author={Fukaya, Kenji},
  title={Floer homology for families---a progress report},
  conference={ title={Integrable systems, topology, and physics}, address={Tokyo}, date={2000}, },
  book={ series={Contemp. Math.}, volume={309}, publisher={Amer. Math. Soc., Providence, RI}, },
  date={2002},
  pages={33--68},
}

\bib{FOOO-can}{article}{
	title={Canonical models of filtered $A_{\infty}$-algebras and {M}orse complexes},
	author={Fukaya, Kenji},
	author={Oh, Yong-Geun},
	author={Ohta, Hiroshi},
	author={Ono, Kaoru},
	journal={New perspectives and challenges in symplectic field theory},
	volume={49},
	pages={201--227},
	year={2009}
}

\bib{FOOO-T1}{article}{
	author={Fukaya, Kenji},
	author={Oh, Yong-Geun},
	author={Ohta, Hiroshi},
	author={Ono, Kaoru},
	title={Lagrangian {F}loer theory on compact toric manifolds. {I}},
	Journal = {Duke Math. J.},
	Number = {1},
	Pages = {23--174},
	Volume = {151},
	Year = {2010},
}

\bib{FOOO-T2}{article}{
	author={Fukaya, Kenji},
	author={Oh, Yong-Geun},
	author={Ohta, Hiroshi},
	author={Ono, Kaoru},
	title={Lagrangian {F}loer theory on compact toric manifolds {II}: bulk deformations},
	Journal = {Selecta Math. (N.S.)},
	Number = {3},
	Pages = {609--711},
	Volume = {17},
	Year = {2011},
}

\bib{FOOO-MS}{article}{
	author={Fukaya, Kenji},
	author={Oh, Yong-Geun},
	author={Ohta, Hiroshi},
	author={Ono, Kaoru},
    TITLE = {Lagrangian {F}loer theory and mirror symmetry on compact toric
	manifolds},
	JOURNAL = {Ast\'{e}risque},
    NUMBER = {376},
	YEAR = {2016},
	PAGES = {vi+340},
}

\bib{Goldstein}{article}{
	AUTHOR = {Goldstein, Edward},
	TITLE = {Calibrated fibrations on noncompact manifolds via group
		actions},
	JOURNAL = {Duke Math. J.},
	VOLUME = {110},
	YEAR = {2001},
	NUMBER = {2},
	PAGES = {309--343},
}

\bib{Gross-eg}{incollection}{
	AUTHOR = {Gross, Mark},
	TITLE = {Examples of special {L}agrangian fibrations},
	BOOKTITLE = {Symplectic geometry and mirror symmetry ({S}eoul, 2000)},
	PAGES = {81--109},
	PUBLISHER = {World Sci. Publ., River Edge, NJ},
	YEAR = {2001},
}

\bib{GS1}{article}{
  author={Gross, Mark},
  author={Siebert, Bernd},
  title={From real affine geometry to complex geometry},
  journal={Annals of Mathematics. Second Series},
  volume={174},
  year={2011},
  number={3},
  pages={1301--1428},
}

\bib{H2}{article}{
  author={Hitchin, Nigel J.},
  title={The moduli space of special Lagrangian submanifolds},
  note={Dedicated to Ennio De Giorgi},
  journal={Ann. Scuola Norm. Sup. Pisa Cl. Sci. (4)},
  volume={25},
  date={1997},
  number={3-4},
  pages={503--515 (1998)},
}

\bib{HKL}{article}{
	author = {Hong, Hansol},
	author = {Kim, Yoosik},
	author = {Siu-Cheong Lau},
	title = {Immersed two-spheres and {SYZ} with application to {Grassmannians}},
	journal = {preprint},
	note = {\href{https://arxiv.org/abs/1805.11738}{arXiv:1805.11738}}
}

\bib{HKLZ}{article}{
	author = {Hong, Hansol},
	author = {Kim, Yoosik},	
	author={Lau, Siu-Cheong},
	author={Zheng, Xiao},
	title={$T$-equivariant disc potentials for toric {C}alabi--{Y}au manifolds},
	journal={preprint},
	note = {\href{https://arxiv.org/abs/1912.11455}{https://arxiv.org/abs/1912.11455}},
}

\bib{KS1}{article}{
  author={Kontsevich, Maxim},
  author={Soibelman, Yan},
  title={Affine structures and non-{A}rchimedean analytic spaces},
  booktitle={The unity of mathematics},
  series={Progr. Math.},
  volume={244},
  pages={321--385},
  publisher={Birkh\"{a}user Boston},
  address={Boston, MA},
  year={2006},
}

\bib{L1}{article}{
	author={Lin, Yu-Shen},
	title={Open Gromov-Witten invariants on elliptic K3 surfaces and
		wall-crossing},
	journal={Comm. Math. Phys.},
	volume={349},
	date={2017},
	number={1},
	pages={109--164},
}

\bib{L2}{article}{
	author={Lin, Yu-Shen},
	title={Correspondence theorem between holomorphic discs and tropical discs on K3 surfaces},
	journal={to appear in J. Diff. Geom.},
	note = {\href{https://arxiv.org/abs/1703.00411}{https://arxiv.org/abs/1703.00411}},
}


\bib{LZ}{article}{
	author={Lau, Siu-Cheong},
	author={Zheng, Xiao},
	title={$T$-equivariant disc potential and SYZ mirror construction},
	journal={preprint},
	note = {\href{https://arxiv.org/abs/1906.11749}{https://arxiv.org/abs/1906.11749}},
}

\bib{L14}{article}{
  author={Lin, Yu-Shen},
  title={Open Gromov--Witten of Del Pezzo Surfaces},
 	journal={preprint},
 note = {\href{https://arxiv.org/abs/2005.08681}{https://arxiv.org/abs/2005.08681}},
}

\bib{LYZ}{article}{
  author={Leung, Naichung Conan},
  author={Yau, Shing-Tung},
  author={Zaslow, Eric},
  title={From special {L}agrangian to {H}ermitian--{Y}ang--{M}ills via {F}ourier--{M}ukai transform},
  journal={Adv. Theor. Math. Phys.},
  volume={4},
  year={2000},
  number={6},
  pages={1319--1341},
}

\bib{Oh}{incollection}{
	AUTHOR = {Oh, Yong-Geun},
	TITLE = {Relative {F}loer and quantum cohomology and the symplectic
		topology of {L}agrangian submanifolds},
	BOOKTITLE = {Contact and symplectic geometry ({C}ambridge, 1994)},
	SERIES = {Publ. Newton Inst.},
	VOLUME = {8},
	PAGES = {201--267},
	PUBLISHER = {Cambridge Univ. Press, Cambridge},
	YEAR = {1996},
}

\bib{P}{article}{
   Author={Pierrick Bousseau},
   title={A proof of N.~Takahashi's conjecture for $(\mathbb{P}^2,E)$ and a refined sheaves/Gromov--Witten correspondence},
   journal={preprint},
   note={\href{https://arxiv.org/abs/1909.02992}{https://arxiv.org/abs/1909.02992}},
   }

\bib{PT}{article}{
	Author={Pascaleff, James},
	Author = {Tonkonog, Dmitry},
	title={The wall-crossing formula and {L}agrangian mutations},
	journal={preprint},
	note = {\href{https://arxiv.org/abs/1711.03209}{https://arxiv.org/abs/1711.03209}},
}


\bib{Seidel}{article}{
	Author = {Seidel, Paul},
	Journal = {preprint},
	Note = {\href{http://math.mit.edu/~seidel/937/lecture-notes.pdf}{http://math.mit.edu/~seidel/937/lecture-notes.pdf}},
	Title = {Lectures on categorical dynamics and symplectic topology}}

\bib{SYZ}{article}{
  author={Strominger, Andrew},
  author={Yau, Shing-Tung},
  author={Zaslow, Eric},
  title={Mirror symmetry is {$T$}-duality},
  journal={Nuclear Phys. B},
  volume={479},
  year={1996},
  number={1-2},
  pages={243--259},
}

\bib{T4}{article}{
  author={Tu, Junwu},
  title={On the reconstruction problem in mirror symmetry},
  journal={Adv. Math.},
  volume={256},
  date={2014},
  pages={449--478},
}

\bib{TY1}{article}{
  author={Tian, G.},
  author={Yau, Shing-Tung},
  title={Complete {K}\"{a}hler manifolds with zero {R}icci curvature. I},
  journal={J. Amer. Math. Soc.},
  volume={3},
  date={1990},
  number={3},
  pages={579--609},
  issn={0894-0347},
  review={\MR {1040196 (91a:53096)}},
  doi={10.2307/1990928},
}

\end{biblist}
\end{bibdiv}
\end{document}